\providecommand{\drfi}{draft}
\documentclass[oneside,article,\drfi]{memoir}

\usepackage{amsmath}         
\usepackage{amsthm,thmtools,thm-restate} 
\usepackage{enumitem}        

\usepackage{tikz}
  \usetikzlibrary{matrix,arrows}
  
\usepackage{tikz-cd}

\usepackage{float} 

\usepackage[nobysame,alphabetic,initials]{amsrefs} 
\usepackage[hidelinks]{hyperref}                   
\usepackage[capitalize,nosort]{cleveref}           

\usepackage{amssymb}
\usepackage[mathscr]{euscript}
\usepackage{relsize}
\usepackage{stmaryrd}
\usepackage{stackengine}
\usepackage{graphicx} 

\usepackage{indentfirst} 
\usepackage{multicol}

\usepackage[inline,nomargin]{fixme} 
  \fxsetup{targetlayout=color}

\usepackage{float} 

\isopage[12]

\setlrmargins{*}{*}{1}
\checkandfixthelayout

\counterwithout{section}{chapter}
\usepackage{appendix}

  \newcommand{\adj}{\dashv}

  \newcommand{\op}{{\mathord\mathrm{op}}}

  \newcommand{\push}{\cup}

  \newcommand{\colim}{\operatornamewithlimits{colim}}

  \newcommand{\slice}{\mathbin\downarrow}

  \newcommand{\fslice}{\mathbin\twoheaddownarrow}

  \newcommand{\bd}{\partial}

  \newcommand{\gprod}{\otimes}

  \newcommand{\id}[1][]{\operatorname{id}_{#1}}
  
  \newcommand{\simp}[1]{\mathord{\Delta^{#1}}}

  \makeatletter
  \def\horn#1{\expandafter\horn@i#1,,\@nil}
  \def\horn@i#1,#2,#3\@nil{\Lambda^{#2}[#1]}
  \makeatother

  \newcommand{\uvar}{\mathord{\relbar}}

  \renewcommand{\tilde}{\widetilde}

  \renewcommand{\hat}{\widehat}

  \newcommand{\cat}[1]{\mathcal{#1}}

  \newcommand{\ncat}[1]{\mathsf{#1}}

  \newcommand{\Simp}{\Delta}

  \newcommand{\sSet}{\ncat{sSet}}

  \newcommand{\from}{\colon}

  \newcommand{\ito}{\hookrightarrow}

  \declaretheorem[style=definition,within=section]{definition}
  \declaretheorem[style=definition,numberlike=definition]{example}
  \declaretheorem[style=definition,numberlike=definition]{remark}

  \declaretheorem[style=plain,numberlike=definition]{corollary}
  \declaretheorem[style=plain,numberlike=definition]{lemma}
  \declaretheorem[style=plain,numberlike=definition]{proposition}
  \declaretheorem[style=plain,numberlike=definition]{theorem}

  \declaretheorem[style=plain,numbered=no,name=Theorem]{theorem*}

  \Crefname{corollary}{Corollary}{Corollaries}
  \Crefname{definition}{Definition}{Definitions}
  \Crefname{lemma}{Lemma}{Lemmas}
  \Crefname{proposition}{Proposition}{Propositions}
  \Crefname{remark}{Remark}{Remarks}
  \Crefname{theorem}{Theorem}{Theorems}

  \newlist{axioms}{enumerate}{1}
  \Crefname{axiomsi}{}{}

  \newenvironment{tikzeq*}
  {
    \begingroup
    \begin{equation*}
    \begin{tikzpicture}[baseline=(current bounding box.center)]
  }
  {
    \end{tikzpicture}
    \end{equation*}
    \endgroup
    \ignorespacesafterend
  }

  \tikzset
  {
    diagram/.style=
    {
      matrix of math nodes,
      column sep={4.3em,between origins},
      row sep={4em,between origins},
      text height=1.5ex,
      text depth=.25ex
    },
    over/.style={preaction={draw=white,-,line width=6pt}},
    every to/.style={font=\footnotesize},
    inj/.style={right hook->},
    surj/.style={-{Latex[open]}},
    cof/.style={>->},
    fib/.style={->>},
  }

  \DeclareFontFamily{U}{mathx}{\hyphenchar\font45}

  \DeclareFontShape{U}{mathx}{m}{n}{
    <5> <6> <7> <8> <9> <10>
    <10.95> <12> <14.4> <17.28> <20.74> <24.88>
    mathx10}{}

  \DeclareSymbolFont{mathx}{U}{mathx}{m}{n}

  \DeclareFontFamily{U}{mathb}{\hyphenchar\font45}

  \DeclareFontShape{U}{mathb}{m}{n}{
    <5> <6> <7> <8> <9> <10>
    <10.95> <12> <14.4> <17.28> <20.74> <24.88>
    mathb10}{}

  \DeclareSymbolFont{mathb}{U}{mathb}{m}{n}

  \DeclareMathAccent{\widebar}{0}{mathx}{"73}

  \DeclareMathSymbol{\Rsh}{\mathrel}{mathb}{"E9}

  \DeclareFontFamily{U}{MnSymbolA}{}

  \DeclareFontShape{U}{MnSymbolA}{m}{n}{
    <-6> MnSymbolA5
    <6-7> MnSymbolA6
    <7-8> MnSymbolA7
    <8-9> MnSymbolA8
    <9-10> MnSymbolA9
    <10-12> MnSymbolA10
    <12-> MnSymbolA12}{}

  \DeclareSymbolFont{MnSyA}{U}{MnSymbolA}{m}{n}

  \DeclareMathSymbol{\twoheaddownarrow}{\mathrel}{MnSyA}{27}

  \newcommand{\MSC}[1]{%
    \let\thempfn\relax
    \footnotetext[0]{2020 Mathematics Subject Classification: #1.}
  }

\newcommand{\cSet}{\mathsf{cSet}} 
\newcommand{\Grp}{\mathsf{Grp}}

\newcommand{\Set}{\mathsf{Set}}
\newcommand{\Kan}{\mathsf{Kan}}

\newcommand{\fcat}[2]{{#2}^{#1}} 
\newcommand{\arr}[1]{\fcat{[1]}{#1}} 
\newcommand{\carr}[1]{\fcat{[2]}{#1}} 
\newcommand{\Top}{\ncat{Top}}
\newcommand{\adjunct}[4]{#1 \from #3 \rightleftarrows #4 : \! #2} 
\newcommand{\Mono}[1]{{#1}_{\mathsf{2}}} 

\newcommand{\boxcat}{\mathord{\square}} 
\newcommand{\face}[2]{\partial^{#1}_{#2}} 
\newcommand{\degen}[2]{\sigma^{#1}_{#2}} 
\newcommand{\conn}[2]{\gamma^{#1}_{#2}} 
\newcommand{\cube}[1]{\mathord{\square^{#1}}} 
\newcommand{\obox}[2]{\mathord{\sqcap^{#1}_{#2}}} 
\newcommand{\dfobox}[1][n]{\mathord{\sqcap^{#1}_{i,\varepsilon}}} 
\newcommand{\reali}[2][]{\lvert #2 \rvert_{#1}} 
\newcommand{\pp}{\mathrel{\hat{\gprod}}} 
\newcommand{\pe}{\mathrel{\triangleright}}
\newcommand{\lhom}[1][L]{\operatorname{hom}_{#1}} 
\newcommand{\rhom}{\operatorname{hom}_{R}} 

\newcommand{\Sing}{\operatorname{Sing}} 
\newcommand{\sSing}{\operatorname{Sing}_{\Simp}} 
\newcommand{\cSing}{\operatorname{Sing}_{\boxcat}} 
\newcommand{\sreali}[1]{\lvert #1 \rvert_{\Simp}} 
\newcommand{\creali}[1]{\lvert #1 \rvert_{\boxcat}} 

\newcommand{\rsim}{\sim} 
\newcommand{\loopl}[1][]{\Omega_L^{#1}} 
\newcommand{\loopr}[1][]{\Omega_R^{#1}} 
\newcommand{\loopsp}[1][]{\Omega^{#1}} 
\newcommand{\hcirc}{\mathrel{\leftrightarrow}} 
\newcommand{\vcirc}{\mathrel{\updownarrow}} 
\newcommand{\cpsq}[1]{\langle #1 \rangle} 

\newcommand{\degAn}[1][n]{DA_n} 

\newcommand{\pob}[1]{P #1} 
\newcommand{\rpath}[1]{\rhom (\cube{1}, #1)} 
\newcommand{\lpath}[1]{\lhom (\cube{1}, #1)} 

\newcommand{\ldrop}[1]{{#1}^{r}} 
\newcommand{\rdrop}[1]{{#1}^{\ell}} 
\newcommand{\FLsect}[1]{\mathsf{const}_{#1}} 
\newcommand{\FLfib}[1]{{#1}_{\ell}} 
\newcommand{\FLfibr}[1]{{#1}_{r}} 
\newcommand{\Q}{\ensuremath{Q_2}} 
\newcommand{\twosp}{\ensuremath{\cube{1} \cup \cube{1}}} 

\newcommand{\restr}[2]{{#1}|_{#2}} 
\newcommand{\noproof}{\hfill\qedsymbol}
\newcommand{\pt}[1]{{#1}_*} 

\thanksmarkseries{arabic}

\author{
  Daniel Carranza \thanks{Johns Hopkins University, Baltimore, United States}
  \and Krzysztof Kapulkin \thanks{University of Western Ontario, London, Ontario, Canada; \textit{corresponding author}, email: \texttt{kkapulki@uwo.ca}} }
\title{Homotopy groups of cubical sets}
\date{\today}

\begin{document}

  \maketitle

  \begin{abstract}
    We define and study homotopy groups of cubical sets.
    To this end, we give four definitions of homotopy groups of a cubical set, prove that they are equivalent, and further that they agree with their topological analogues via the geometric realization functor.
    We also provide purely combinatorial proofs of several classical theorems, including: product preservation, commutativity of higher homotopy groups, the long exact sequence of a fibration, and Whitehead's theorem.
    
    This is a companion paper to our ``Cubical setting for discrete homotopy theory, revisited'' in which we apply these results to study the homotopy theory of simple graphs.
     \MSC{55Q05, 18N40 (primary), 55U35 (secondary)}
  \end{abstract}

\section*{Introduction}

Cubical sets, introduced by Kan \cite{kan:abstract-homotopy-i,kan:abstract-homotopy-ii}, provide a convenient combinatorial model for the homotopy theory of spaces.
A cubical set is a presheaf on the category $\Box$ of combinatorial cubes (i.e.~a functor $\Box^\op \to \Set$), just as a simplicial set is a presheaf on the category $\Delta$ of combinatorial simplicies (a functor $\Delta^\op \to \Set$, respectively).
While combinatorial simplices can be viewed as iterated joins of the interval, combinatorial cubes are an axiomatization of iterated products of the interval.

The homotopy theory of simplicial sets is now well developed, with a detailed account given for example in \cite{goerss-jardine}.
In this paper, we take a first step towards establishing a similar theory for cubical sets by defining homotopy groups of cubical sets and proving classical theorems about them.
Specifically, we define homotopy groups of cubical Kan complexes and give equivalent characterizations thereof:
\begin{enumerate}
  \item $\pi_n (X, x)$ is the set of path components of the Kan complex $\Omega^n(X, x)$ of the $n$-iterated loop space of $(X, x)$;
  \item $\pi_n (X, x)$ is the set of homotopy classes of maps $\cube{n} \to X$ from the representable to $X$ that take the boundary to $x$;
  \item $\pi_n (X, x)$ is the set of pointed homotopy classes of pointed maps $\cube{n}/\bd \cube n \to X$;
  \item $\pi_n (X, x)$ is the set of pointed homotopy classes of pointed maps $\bd \cube {n+1} \to X$,
\end{enumerate}
where for (1), we note that the empty cubical set has no path components, i.e.~$\pi_0(\varnothing) = \varnothing$.
We further establish that these homotopy groups agree with their analogues for topological spaces via the geometric realization functor.
We then go on to prove several classical results: preservation of products; commutativity of $\pi_n(X, x)$ for $n \geq 2$; the long exact sequence of a fibration; and Whitehead's theorem.

Our development relies on techniques from abstract homotopy theory, specifically model categories \cite{quillen:homotopical-algebra} and fibration categories in the sense of Brown \cite{brown}.
We take the existence of the Grothendieck model category structure on the category of cubical sets, established in \cite{cisinski:presheaves}, as a starting point of our work.

This paper originates from our work \cite{carranza-kapulkin:cubical-setting} on a cubical framework for discrete homotopy theory, which is a homotopy theory of simple graphs \cite{barcelo-kramer-laubenbacher-weaver,babson-barcelo-longueville-laubenbacher}, and a realization that many results about homotopy groups of graphs can be reduced to similar statements about homotopy groups of certain cubical sets associated to graphs.
To take proper advantage of this framework, we therefore need a robust combinatorial theory of homotopy groups of cubical sets developed in the present paper.
While some of our results are certainly known in folklore and would be familiar to experts, their use in work that is primarily on other topics, such as combinatorics, could be more controversial.
To provide a firm mathematical footing for the use of such results, we included a number of them in this paper.

This goal also informs our choice of references when we need to rely on results from other sources.
Rather than referencing the original author, we choose to reference newer and, often more complete, accounts.
For instance, when discussing background on model categories, we usually use Hovey's monograph \cite{hovey} on the topic, rather than Quillen's original work \cite{quillen:homotopical-algebra}.
For many facts on the model structure on cubical sets, we reference \cite{doherty-kapulkin-lindsey-sattler}, rather than the earlier work of Cisinski \cite{cisinski:presheaves,cisinski:elegant}, Jardine \cite{jardine:a-beginning}, and Maltsiniotis \cite{maltsiniotis}.

We assume that our cubical sets have connections.
Connections, introduced by Brown and Higgins \cite{brown-higgins}, are an extra degeneracy operation, making the theory of cubical sets better behaved.
Most examples of cubical sets which appear in practice (such as the cubical singular complex of a topological space) come naturally equipped with connections.
The role of connections in this paper is to simplify arguments of certain elementary results, namely that the $n$-cube is contractible (\cref{thm:cube_n_contr}) and that $\pi_n(X, x)$ has a unit element (\cref{pi-1-is-a-group}).
By assuming connections, we are able to give shorter proofs which do not require passing to simplicial sets or topological spaces.
While connections are not required to define and develop cubical homotopy groups, it is unclear whether the resulting theory would be as elegant or self-contained.

More precisely, we work with cubical sets equipped with both positive (i.e.~$\min$) and negative (i.e.~$\max$) connections, however all proofs apply verbatim to the case of only negative connections, as considered by Cisinski \cite{cisinski:elegant} and Maltsiniotis \cite{maltsiniotis}.
If one prefers to work with only positive connections, a minor adjustment of definitions and proofs is required, but the underlying arguments remain valid for this choice of the cube category.

One difficulty in working with cubical sets is that the Cartesian product is not well-behaved.
For instance, in the category of cubical sets without connections, the Cartesian product of the interval (i.e.~the representable $\cube{1}$) with itself has the homotopy type of $S^2 \vee S^1$.
Connections go partway towards addressing this problem --- in the category of cubical sets with connections, although this product is contractible, it is still not isomorphic to the combinatorial square (i.e.~the representable $\cube{2}$).
To address this issue, one works with the geometric product instead.
The geometric product is a different monoidal structure on the category of cubical sets defined via left Kan extension with the property that the geometric product of an $m$-cube and an $n$-cube is, by definition, an $(m+n)$-cube.
Moreover, this product is homotopically well-behaved in that the geometric realization sends it to the Cartesian product of topological spaces.
The advantages of the geometric product cannot be overstated --- a number of our proofs including the construction of loop spaces admit a particularly simple description thanks to the properties of the geometric product.

\textbf{Organization of the paper.}
The paper is organized as follows.
\cref{sec:model-cat} contains the background on abstract homotopy theory: we collect the necessary definitions and results on model categories and fibration categories that we will be using throughout the paper.
In \cref{sec:cset}, we introduce cubical sets and their homotopy theory.
A large part of this material can be found in the literature or is known in folklore, but we give proofs of all results that may be hard to find.
In \cref{sec:groups}, we define homotopy groups of cubical sets via path components of the iterated loop space, and prove three equivalent characterizations thereof.
Finally, in \cref{sec:results}, we prove classical results: preservation of products; commutativity of $\pi_n(X, x)$ for $n \geq 2$; the long exact sequence of a fibration; and Whitehead's theorem.
Our proof of the long exact sequence follows Mather's proof \cite{mather:pullbacks}  (cf.~\cite[Thm.~3.8.12]{cisinski:higher-categories}) in the case of spaces and relies on the theory of homotopy pullbacks.
As this theory is well-known in folklore but perhaps not well-documented in literature, we dedicate \cref{sec:hopullbacks} to establishing basic facts about homotopy pullbacks: first in fibration categories, and later specifically in the category of (cubical) Kan complexes.

\textbf{Prerequisites.}
The only prerequisite expected of the reader is familiarity with category theory, for which \cite{maclane:categories} and \cite{riehl:context} are canonical references.

\textbf{Acknowledgements.}
We thank Karol Szumi{\l}o for numerous helpful conversations, Denis-Charles Cisinski for catching a mistake in an earlier draft of this paper, and the anonymous referee for many helpful comments that helped improve the presentation of these results.
During the work on this paper, the first author was partially supported by the NSERC Undergraduate Student Research Award, and the second author was partially supported by the NSERC Discovery Grant.
We thank NSERC for this support.

\section{Abstract homotopy theory} \label{sec:model-cat}

In this section, we review the necessary background on abstract homotopy theory.
We consider two such frameworks: model categories and fibration categories.

\subsection{Model categories}

Model category theory provides a convenient framework for the study of different homotopy theories and comparisons between them.
Our presentation largely follows \cite{hovey}.

\begin{definition}
    A \emph{model category} is a (co)complete category $\cat{C}$ along with three subcategories of \emph{cofibrations}, \emph{fibrations}, and \emph{weak equivalences} such that (in what follows, an \emph{acyclic cofibration} is a map which is both a cofibration and a weak equivalence whereas an \emph{acyclic fibration} is a map which is both a fibration and a weak equivalence):
    \begin{enumerate}
        \item weak equivalences satisfy two-out-of-three;
        \item all three subcategories are closed under retracts;
        \item every isomorphism of $\cat{C}$ is in all three subcategories;
        \item every map in $\cat{C}$ factors both as a cofibration followed by an acyclic fibration and as an acyclic cofibration followed by a fibration;
        \item acyclic cofibrations have the left lifting property with respect to fibrations (i.e.~any commutative square
        \[ \begin{tikzcd}
            \cdot \ar[d, "f"'] \ar[r] & \cdot \ar[d, "g"] \\
            \cdot \ar[r] & \cdot
        \end{tikzcd} \]
        where $f$ is an acyclic cofibration and $g$ is a fibration admits a lift);
        \item cofibrations have the left lifting property with respect to acyclic fibrations.
    \end{enumerate}
\end{definition}
\begin{example}[{\cite[Thm.~2.4.19]{hovey},\cite[Prop.~8.3]{dwyer-spalinski}}]
    The category $\Top$ of topological spaces is a model category where
    \begin{itemize}
        \item cofibrations are retracts of generalized cell inclusions (cf.~\cite[Rem.~8.8 \& Prop.~8.9]{dwyer-spalinski});
        \item fibrations are Serre fibrations;
        \item weak equivalences are maps $f \from X \to Y$ which induce both a bijection $\pi_0 X \to \pi_0 Y$ on path components and an isomorphism $\pi_n(X,x) \to \pi_n(Y,f(x))$ for all $x \in X$ and $n \geq 1$.
    \end{itemize}
\end{example}
\begin{example}[{\cite[Prop.~2.1.5]{cisinski:presheaves}, \cite[Thm.~3.6.5]{hovey}}] \label{ex:geometric-realization}
    The category $\sSet$ of simplicial sets (i.e.~functors $\Simp^\op \to \Set$) is a model category where
    \begin{itemize}
        \item cofibrations are monomorphisms;
        \item fibrations are maps which have the right lifting property with respect to all horn inclusions;
        \item weak equivalences are maps $X \to Y$ such that, for any Kan complex $Z$, the precomposition map $[Y, Z] \to [X, Z]$ is a bijection on homotopy classes of maps.
    \end{itemize}
    Readers familiar with the work of Quillen \cite{quillen:homotopical-algebra} may find this description of weak equivalences somewhat surprising, as they are often defined as maps inducing isomorphisms on geometric realizations (see \cref{ex:realization} below).
    The two definitions are equivalent by \cite[Lem.~I.4.1]{goerss-jardine}, \cite[Thm.~1.2.10.ii]{hovey}, and the Yoneda lemma.
\end{example}
\begin{example}[{\cite[pg.~5]{hovey}}] \label{slice_model_structure}
    Given a model category $\cat{M}$ and an object $A \in \cat{M}$, the slice category $A \downarrow \cat{M}$ of maps $A \to X$ in $\cat{M}$ is a model category where all three classes are created by the projection functor $A \downarrow \cat{M} \to \cat{M}$.
    That is, a morphism in $A \downarrow \cat{M}$ is a cofibration/fibration/weak equivalence if the underlying map in $\cat{M}$ is.
\end{example}

When drawing a diagram in a model category, we occasionally depict a cofibration by a hooked arrow $X \ito Y$, a fibration by a two-headed arrow $X \twoheadrightarrow Y$, and a weak equivalence by an arrow $X \xrightarrow{\sim} Y$ labelled with a tilde.

We establish the following notions for objects in a model category.
\begin{definition}
    For a model category $\cat{M}$, let $\varnothing \in \cat{M}$ be initial and $1 \in \cat{M}$ be terminal.
    \begin{enumerate}
        \item An object $X$ is \emph{cofibrant} if the map $\varnothing \to X$ is a cofibration.
        \item An object $X$ is \emph{fibrant} if the map $X \to 1$ is a fibration.
    \end{enumerate} 
\end{definition}

The model categories we consider have \emph{functorial factorizations}.
\begin{definition}[{\cite[Def.~3.2.1]{riehl-verity}}] \label{func_fact}
    For a model category $\cat{M}$, a \emph{functorial factorization} on $\cat{M}$ consists of two functors $Q, R \from \arr{\cat{M}} \to \carr{\cat{M}}$ from the category of arrows in $\cat{M}$ to the category of composable pairs of arrows in $\cat{M}$ such that, for a morphism $f \in \cat{C}$,
    \begin{enumerate}
        \item the pair $Qf \in \carr{\cat{M}}$ is a factorization of $f$ as a cofibration followed by an acyclic fibration;
        \item the pair $Rf \in \carr{\cat{M}}$ is a factorization of $f$ as an acyclic cofibration followed by a fibration.
    \end{enumerate}
\end{definition}
Explicitly, if two morphisms, $f \from X \to Y$ and $g \from Z \to W$, in a model category are the horizontal components of a commutative square
\[ \begin{tikzcd}
    X \ar[r, "f"] \ar[d] & Y \ar[d] \\
    Z \ar[r, "g"] & W
\end{tikzcd} \]
then a functorial factorization produces the following two commutative diagrams.
\[ \begin{tikzcd}
    X \ar[r, hook] \ar[rr, bend left, "f"] \ar[d] & \tilde{Y} \ar[r, two heads, "\sim"] \ar[d] & Y \ar[d] \\
    Z \ar[r, hook] \ar[rr, bend right, "g"] & \tilde{W} \ar[r, two heads, "\sim"] & W
\end{tikzcd} \qquad \begin{tikzcd}
    X \ar[r, hook, "\sim"] \ar[rr, bend left, "f"] \ar[d] & \tilde{X} \ar[r, two heads] \ar[d] & Y \ar[d] \\
    Z \ar[r, hook, "\sim"] \ar[rr, bend right, "g"] & \tilde{Z} \ar[r, two heads] & W
\end{tikzcd} \]
The diagram on the left is given by applying $Q$ to the starting square (viewed as a morphism in the arrow category).
The diagram on the right is given by applying $R$ to the starting square.
In particular, the functor $Q$ produces a morphism $\tilde{Y} \to \tilde{W}$ which makes the diagram commute.
Likewise, the functor $R$ produces a morphism $\tilde{X} \to \tilde{Z}$ which makes the diagram commute.

An advantage of functorial factorizations is that they give both a \emph{cofibrant replacement} functor and a \emph{fibrant replacement} functor.
Given an object $X \in \cat{M}$ in a model category, applying $Q$ to the canonical map $\varnothing \ito X$ yields a pair of morphisms $\varnothing \ito \tilde{X} \overset{\sim}{\twoheadrightarrow} X$.
Restricting to the middle object in this factorization gives a functor $Q \from \cat{M} \to \cat{M}$ which takes values in cofibrant objects with a natural weak equivalence (i.e.~a natural transformation whose components are weak equivalences) $Q \overset{\sim}{\Rightarrow} \id[\cat{M}]$.
We refer to such a functor as a \emph{cofibrant replacement} functor.
Likewise, applying $R$ to the canonical map $X \to 1$ induces a functor $R \from \cat{M} \to \cat{M}$ which takes values in fibrant objects with a natural weak equivalence $\id[\cat{M}] \overset{\sim}{\Rightarrow} R$.
We refer to such a functor as a \emph{fibrant replacement} functor.

The suitable notion of a morphism between model categories is a Quillen adjunction.
\begin{definition}
    For two model categories $\cat{M}, \cat{M'}$, a \emph{Quillen adjunction} is an adjunction $\adjunct{F}{G}{\cat{M}}{\cat{M'}}$ such that the left adjoint $F$ preserves cofibrations and acyclic cofibrations.
\end{definition}
\begin{proposition}
    For an adjunction $(F \adj G)$, the following are equivalent:
    \begin{enumerate}
        \item $(F \adj G)$ is a Quillen adjunction;
        \item the right adjoint $G$ preserves fibrations and acyclic fibrations;
        \item the left adjoint $F$ preserves cofibrations and the right adjoint $G$ preserves fibrations;
        \item the left adjoint $F$ preserves acyclic cofibrations and the right adjoint $G$ preserves acyclic fibrations.
    \end{enumerate}
\end{proposition}
\begin{proof}
    This follows from the argument given in the proof of \cite[Lem.~1.3.4]{hovey}.
\end{proof}
\begin{example}[{\cite[Thm.~3.6.7]{hovey}}] \label{ex:realization}
    The topological simplices form a functor $\Simp \to \Top$.
    The left Kan extension of this functor along the Yoneda embedding $\Delta \to \sSet$ gives the \emph{geometric realization} functor $\sreali{\uvar} \from \sSet \to \Top$. Its right adjoint is the \emph{simplicial singular complex} functor $\sSing \from \Top \to \sSet$.
\end{example}
\begin{example}
    For a model category $\cat{M}$ and $A \in \cat{M}$, the projection functor $A \downarrow \cat{M} \to \cat{M}$ has a left adjoint $\uvar \sqcup A \from \cat{M} \to A \downarrow \cat{M}$ given by taking the coproduct with $A$.
    This gives a Quillen adjunction $\cat{M} \rightleftarrows A \downarrow \cat{M}$ where $A \downarrow \cat{M}$ has the model structure induced by $\cat{M}$ as described in \cref{slice_model_structure}.
\end{example}
\begin{definition}
    A Quillen adjunction $\adjunct{F}{G}{\cat{M}}{\cat{M'}}$ is a \emph{Quillen equivalence} if, for every cofibrant $X \in \cat{M}$ and fibrant $Y \in \cat{M'}$, the adjunction bijection $\cat{M}(FX, Y) \xrightarrow{\cong} \cat{M'}(X, GY)$ preserves and reflects weak equivalences.
\end{definition}
In practice, we will use the following characterization to verify when a Quillen adjunction is a Quillen equivalence.
\begin{proposition}[{\cite[Prop.~1.3.13]{hovey}}]
    A Quillen adjunction $\adjunct{F}{G}{\cat{M}}{\cat{M'}}$ is a Quillen equivalence if and only if
    \begin{enumerate}
        \item the composite natural map $X \to GF(X) \to GRF(X)$ is a weak equivalence for all cofibrant $X \in \cat{M}$; and
        \item the composite natural map $FQG(Y) \to FG(Y) \to Y$ is a weak equivalence for all fibrant $Y \in \cat{M'}$. \qed
    \end{enumerate}
\end{proposition}
\begin{example}[{\cite[Thm.~3.6.7]{hovey}}]
    The geometric realization and simplicial singular complex adjunction $\adjunct{\sreali{\uvar}}{\sSing}{\sSet}{\Top}$ is a Quillen equivalence.
\end{example}

Fix a model category $\cat{M}$. 
We introduce a notion of homotopy for maps in a model category.
\begin{definition}
    For $X \in \cat{M}$, 
    \begin{enumerate}
        \item a \emph{cylinder object} on $X$ is a factorization of the map $[\id[X], \id[X]] \from X \sqcup X \to X$ as a cofibration $X \sqcup X \to CX$ and a weak equivalence $CX \to X$;
        \item a \emph{path object} on $X$ is a factorization of the map $(\id[X], \id[X]) \from X \to X \times X$ as a weak equivalence $X \to PX$ and a fibration $PX \to X \times X$. 
    \end{enumerate}
\end{definition}
Note that a functorial factorization gives a functorial cylinder object and a functorial path object.
\begin{definition}
    For maps $f, g \from X \to Y$ in a model category $\cat{M}$, 
    \begin{enumerate}
        \item a \emph{left homotopy} from $f$ to $g$ is a map $H \from CX \to Y$ from the cylinder object on $X$ to $Y$ such that the diagram
        \[ \begin{tikzcd}
            X \ar[d, swap, hook, "i_1"] \ar[drr, "f", bend left] \\
            X \sqcup X \ar[r] & CX \ar[r, "H"] & Y \\
            X \ar[u, swap, hook, "i_2"'] \ar[urr, swap, "g", bend right]
        \end{tikzcd} \]
        commutes.
        \item a \emph{right homotopy} from $f$ to $g$ is a map $H \from X \to PY$ from $X$ to the path object on $Y$ such that the diagram 
        \[ \begin{tikzcd}
            && Y \\
            X \ar[r, "H"] \ar[rru, "f", bend left] \ar[rrd, swap, "g", bend right] & PY \ar[r] & Y \times Y \ar[u, swap, "\mathrm{pr}_1"] \ar[d, swap, "\mathrm{pr}_2"'] \\
            && Y
        \end{tikzcd} \]
        commutes.
    \end{enumerate}
\end{definition}
\begin{proposition}[{\cite[Prop.~1.2.5]{hovey}}]
    For $X, Y \in \cat{M}$, if $X$ is cofibrant and $Y$ is fibrant then the notions of left and right homotopy coincide and form an equivalence relation on maps $X \to Y$. \noproof
\end{proposition}
With this, given a cofibrant object $X \in \cat{M}$ and a fibrant object $Y \in \cat{M}$, we refer to both a left and right homotopy between maps $X \to Y$ as simply a \emph{homotopy}.
This notion of homotopy naturally gives rise to a notion of homotopy equivalence.
\begin{definition}
    Let $X, Y \in \cat{M}$ be both fibrant and cofibrant.
    A map $f \from X \to Y$ is a \emph{homotopy equivalence} if there exist
    \begin{enumerate}
        \item a map $g \from Y \to X$;
        \item a homotopy from $gf$ to $\id[X]$;
        \item a homotopy from $fg$ to $\id[Y]$.
    \end{enumerate}
\end{definition}
We have that a map between objects which are fibrant and cofibrant is a homotopy equivalence precisely when it is a weak equivalence.
\begin{theorem}[{\cite[Prop.~1.2.8]{hovey}}] \label{thm:wequiv_is_htpy_equiv}
    Let $X, Y \in \cat{M}$ be fibrant and cofibrant.
    A map $f \from X \to Y$ is a homotopy equivalence if and only if it is a weak equivalence.
    \noproof
\end{theorem}
We apply \cref{thm:wequiv_is_htpy_equiv} to the case of a pointed model category $\pt{\cat{M}}$; that is, the slice model category $1 \downarrow \cat{M}$ under the terminal object (which is an instance of the model category described in \cref{slice_model_structure}).
We write $(X,x)$ for an object $x \from 1 \to X$ in $\pt{\cat{M}}$ and $(X,x) \to (Y,y)$ for a morphism from $(X,x)$ to $(Y,y)$ in $\pt{\cat{M}}$.
\begin{corollary} \label{thm:wequiv_is_htpy_equiv_ptd}
    Let $\pt{\cat{M}}$ be a pointed model category and $(X,x), (Y,y) \in \pt{\cat{M}}$ be fibrant and cofibrant.
    A map $(X,x) \to (Y,y)$ is a homotopy equivalence if and only if the underlying map $X \to Y$ is a weak equivalence in $\cat{M}$. \noproof
\end{corollary}

We make the following statement about weak equivalences in pointed model categories.
\begin{proposition} \label{thm:we_to_hom_iso}
    Let $\pt{\cat{M}}$ be a pointed model category.
    Suppose $(Z, z) \in \pt{\cat{M}}$ is fibrant and $f \from (X,x) \to (Y,y)$ is a map between cofibrant objects.
    If $f$ is a weak equivalence then pre-composition by $f$ induces a bijection $f^* \from [(Y,y),(Z,z)]_* \to [(X,x),(Z,z)]_*$ on pointed homotopy classes of pointed maps.
\end{proposition}
\begin{proof}
    This is an instantiation of \cite[Prop.~1.2.5.iv]{hovey} to the case of the model category $\pt{\cat{M}}$.
\end{proof}

When defining homotopy groups, we will first define them for fibrant objects, and then extend the definition to all objects by taking fibrant replacement.
A priori, such a definition may depend on the choice of fibrant replacement.
The following technical statement shows that any two fibrant replacements are homotopy equivalent, thus showing independence of the choice of fibrant replacement.

\begin{proposition} \label{thm:fibr_repl_equiv}
    Let $X \in \cat{M}$ be cofibrant and $X', X'' \in \cat{M}$ be both fibrant and cofibrant. 
    Given weak equivalences $f \from X \xrightarrow{\sim} X'$ and $g \from X \xrightarrow{\sim} X''$, there exists a homotopy equivalence $X' \to X''$.
\end{proposition}
\begin{proof}
    We factor the map $(f, g) \from X \to X' \times X''$ as an acyclic cofibration $X \overset{\sim}{\ito} \overline{X}$ followed by a fibration $\overline{X} \twoheadrightarrow X' \times X''$.
    This gives the following commutative diagram.
    \[ \begin{tikzcd}
        X \ar[drr, "f", "\sim"', hook] \ar[r, "\sim"] & \overline{X} \ar[r, two heads] & X' \times X'' \ar[d, "\mathrm{pr}_1"] \\
        && X'
    \end{tikzcd} \]
    As $X$ is cofibrant, $\overline{X}$ is.
    As $X' \times X''$ is fibrant, $\overline{X}$ is.
    By 2-out-of-3, the map $\overline{X} \to X'$ is a weak equivalence.
    By \cref{thm:wequiv_is_htpy_equiv}, this map is a homotopy equivalence.
    In particular, it has a homotopy inverse $X' \xrightarrow{\sim} \overline{X}$ which is also a homotopy equivalence.
    
    An analogous argument gives a homotopy equivalence $\overline{X} \to X''$.
    The composite of homotopy equivalences $X' \to \overline{X} \to X''$ thus gives a homotopy equivalence $X' \to X''$.
\end{proof}

\subsection{Fibration categories}

Our second framework for abstract homotopy is that of fibration categories.
They were originally introduced by Brown \cite{brown} under the name \emph{categories of fibrant objects}, and have since been extensively studied by Radulescu-Banu \cite{radulescu-banu}, Cisinski \cite{cisinski:invariance}, and Szumi{\l}o \cite{szumilo:cofib-cat,szumilo:frames,szumilo:cocomplete,kapulkin-szumilo:frames}.
Brown's original motivation came from sheaf cohomology; since then, fibration categories have found many further applications.
These include: the seminal work of Jardine \cite{jardine:presheaves} on model structures on simplicial presheaves (later also adapted to the construction of the Kan--Quillen model structure on simplicial sets \cite{goerss-jardine}), the work of Cisinski on algebraic K-theory \cite{cisinski:invariance}, and more recent results relating fibration categories to type theory \cite{avigad-kapulkin-lumsdaine,kapulkin:lccqc,kapulkin-szumilo:internal-languages}.

\begin{definition}
    A \emph{fibration category} consists of a category $\cat{C}$ together with two wide subcategories (i.e.~containing all objects of $\cat{C}$) of \emph{fibrations} and \emph{weak equivalences} such that (in what follows, an \emph{acyclic fibration} is a map that is both a fibration and a weak equivalence):
    \begin{enumerate}
     \item weak equivalences satisfy two-out-of-three;
     \item all isomorphisms are acyclic fibrations;
     \item pullbacks along fibrations exist; fibrations and acyclic fibrations are stable under pullback;
     \item $\cat{C}$ has a terminal object 1; the canonical map $X \to 1$ is a fibration for any object $X \in \cat{C}$ (that is, all objects are \emph{fibrant});
     \item every map can be factored as a weak equivalence followed by a fibration.
    \end{enumerate}
\end{definition}

   \begin{example}[{\cite[Ex.~1]{brown}}] \label{ex:model_cat_fib_cat}
       For any model category $\cat{M}$, its full subcategory $\cat{M}_{\mathsf{fib}}$ of fibrant objects has a fibration category structure where a map is a fibration/weak equivalence if it is one in the model structure on $\cat{M}$.
   \end{example}

   \begin{example}
       For a fibration category $\cat{C}$ and $X \in \cat{C}$, the full subcategory of the slice category $\cat{C} \slice X$ over $X$ consisting of fibrations is a fibration category, which we denote by $\cat{C} \fslice X$.
   \end{example}
   
   The suitable notion of a functor between fibration categories is an \emph{exact} functor.
   \begin{definition}
    A functor $F \from \cat{C} \to \cat{D}$ between fibration categories is \emph{exact} if it preserves fibrations, acyclic fibrations, pullbacks along fibrations, and the terminal object.
   \end{definition}
   \begin{example}
       If $\adjunct{F}{G}{\cat{M}}{\cat{M'}}$ is a Quillen adjunction then the restriction of the right adjoint $G \from \cat{M'}_{\mathsf{fib}} \to \cat{M}_{\mathsf{fib}}$ to the subcategory of fibrant objects is an exact functor between fibration categories.
   \end{example}
   \begin{example}\label{lem:pb-is-exact}
    Let $f \from X \to Y$ be a morphism in a fibration category $\cat{C}$. Then the functor $f^* \from \cat{C} \fslice Y \to \cat{C} \fslice X$ given by pullback is exact.
   \end{example}

\section{Cubical sets and Kan complexes} \label{sec:cset}

In this section, we collect the necessary facts about cubical sets, using \cite{doherty-kapulkin-lindsey-sattler} as our primary reference.
Other references on the topic include \cite{cisinski:presheaves},\cite{cisinski:elegant}, \cite{jardine:a-beginning}, and \cite{kapulkin-voevodsky}.

\subsection{Cubical sets}
We begin by defining the cube category $\Box$.
The objects of $\Box$ are posets of the form $[1]^n = \{0 \leq 1\}^n$ and the maps are generated (inside the category of posets) under composition by the following four special classes:
\begin{itemize}
  \item \emph{faces} $\partial^n_{i,\varepsilon} \colon [1]^{n-1} \to [1]^n$ for $i = 1, \ldots , n$ and $\varepsilon = 0, 1$ given by:
  \[ \partial^n_{i,\varepsilon} (x_1, x_2, \ldots, x_{n-1}) = (x_1, x_2, \ldots, x_{i-1}, \varepsilon, x_i, \ldots, x_{n-1})\text{;}  \]
  \item \emph{degeneracies} $\sigma^n_i \colon [1]^n \to [1]^{n-1}$ for $i = 1, 2, \ldots, n$ given by:
  \[ \sigma^n_i ( x_1, x_2, \ldots, x_n) = (x_1, x_2, \ldots, x_{i-1}, x_{i+1}, \ldots, x_n)\text{;}  \]
  \item \emph{negative connections} $\gamma^n_{i,0} \colon [1]^n \to [1]^{n-1}$ for $i = 1, 2, \ldots, n-1$ given by:
  \[ \gamma^n_{i,0} (x_1, x_2, \ldots, x_n) = (x_1, x_2, \ldots, x_{i-1}, \max\{ x_i , x_{i+1}\}, x_{i+2}, \ldots, x_n) \text{.} \]
  \item \emph{positive connections} $\gamma^n_{i,1} \colon [1]^n \to [1]^{n-1}$ for $i = 1, 2, \ldots, n-1$ given by:
  \[ \gamma^n_{i,1} (x_1, x_2, \ldots, x_n) = (x_1, x_2, \ldots, x_{i-1}, \min\{ x_i , x_{i+1}\}, x_{i+2}, \ldots, x_n) \text{.} \]
\end{itemize}

These maps obey the following \emph{cubical identities}:

\[ \begin{array}{l l}
    \partial_{j, \varepsilon'} \partial_{i, \varepsilon} = \partial_{i+1, \varepsilon} \partial_{j, \varepsilon'} \quad \text{for } j \leq i; & 
    \sigma_j \partial_{i, \varepsilon} = \begin{cases}
        \partial_{i-1, \varepsilon} \sigma_j & \text{for } j < i; \\
        \id                                                       & \text{for } j = i; \\
        \partial_{i, \varepsilon} \sigma_{j-1} & \text{for } j > i;
    \end{cases} \\
    \sigma_i \sigma_j = \sigma_j \sigma_{i+1} \quad \text{for } j \leq i; &
    \gamma_{j,\varepsilon'} \gamma_{i,\varepsilon} = \begin{cases}
    \gamma_{i,\varepsilon} \gamma_{j+1,\varepsilon'} & \text{for } j > i; \\
    \gamma_{i,\varepsilon}\gamma_{i+1,\varepsilon} & \text{for } j = i, \varepsilon' = \varepsilon;
    \end{cases} \\
    \gamma_{j,\varepsilon'} \partial_{i, \varepsilon} = \begin{cases} 
        \partial_{i-1, \varepsilon} \gamma_{j,\varepsilon'}   & \text{for } j < i-1 \text{;} \\
        \id                                                         & \text{for } j = i-1, \, i, \, \varepsilon = \varepsilon' \text{;} \\
        \partial_{j, \varepsilon} \sigma_j         & \text{for } j = i-1, \, i, \, \varepsilon = 1-\varepsilon' \text{;} \\
        \partial_{i, \varepsilon} \gamma_{j-1,\varepsilon'} & \text{for } j > i;
    \end{cases} &
    \sigma_j \gamma_{i,\varepsilon} = \begin{cases}
        \gamma_{i-1,\varepsilon} \sigma_j  & \text{for } j < i \text{;} \\
        \sigma_i \sigma_i           & \text{for } j = i \text{;} \\
        \gamma_{i,\varepsilon} \sigma_{j+1} & \text{for } j > i \text{.} 
    \end{cases}
\end{array} \]
\begin{definition} \leavevmode
    \begin{enumerate}
        \item A \emph{cubical set} is a functor $\boxcat^\op \to \Set$.
        \item Given cubical sets $X, Y$, a \emph{cubical map} $X \to Y$ is a natural transformation from $X$ to $Y$.
    \end{enumerate}
\end{definition}
We write $\cSet$ for the category of cubical sets and cubical maps.
Given a cubical set $X$, we write $X_n$ for the value of $X$ at $[1]^n$ and write cubical operators on the right e.g.~given an $n$-cube $x \in X_n$ of $X$, we write $x\face{}{1,0}$ for the $\face{}{1,0}$-face of $x$.
\begin{definition}
    For $n \geq 0$,
    \begin{enumerate}
        \item the \emph{combinatorial $n$-cube} $\cube{n}$ is the representable functor $\boxcat(-, [1]^n) \from \boxcat^\op \to \Set$;
        \item the \emph{boundary of the $n$-cube} $\bd \cube{n}$ is the subobject of $\cube{n}$ defined by
        \[ \bd \cube{n} := \bigcup\limits_{\substack{j=1,\dots,n \\ \eta = 0, 1}} \operatorname{Im} \face{}{j,\eta}. \]
        \item given $i = 1, \dots, n$ and $\varepsilon = 0, 1$, the \emph{$(i,\varepsilon)$-open box} $\dfobox$ is the subobject of $\bd \cube{n}$ defined by
        \[ \dfobox := \bigcup\limits_{(j,\eta) \neq (i,\varepsilon)} \operatorname{Im} \face{}{j,\eta}. \]
    \end{enumerate}
\end{definition}

Define a monoidal product $\uvar \gprod \uvar \from \boxcat \times \boxcat \to \boxcat$ on the cube category by $[1]^m \gprod [1]^n = [1]^{m+n}$.
Postcomposing with the Yoneda embedding and left Kan extending gives the \emph{geometric product} of cubical sets.
\[ \begin{tikzcd}
    \boxcat \times \boxcat \ar[r, "\gprod"] \ar[d] & \boxcat \ar[r] & \cSet \\
    \cSet \times \cSet \ar[urr, "\gprod"']
\end{tikzcd} \]
Explicitly, for cubical sets $X$ and $Y$, the geometric product $X \gprod Y$ may be computed as the colimit
\[ X \gprod Y \cong \colim\limits_{\substack{\cube{m} \to X \\ \cube{n} \to Y}} \cube{m+n}. \]
From this, one sees that $\cube{0}$ is the unit of this monoidal product, as the above formula gives
\[ X \gprod \cube{0} \cong \colim\limits_{\substack{\cube{m} \to X \\ \cube{n} \to \cube{0}}} \cube{m+n} \cong \colim\limits_{\cube{n} \to X} \cube{n} \cong X \]
and similarly for $\cube{0} \gprod X$.

Another consequence of this formula is that 0-cubes of the geometric product are indexed by pairs $(x \in X_0 , y \in Y_0)$.
That is,
\[ (X \gprod Y)_0 \cong X_0 \times Y_0. \]
The 1-cubes of $X \gprod Y$ form a pushout of sets
\[ \begin{tikzcd}
    X_0 \times Y_0 \ar[r, "{\degen{}{1} \times \id}"] \ar[d, "{\id \times \degen{}{1}}"'] \ar[rd, phantom, "\ulcorner" very near end] & X_1 \times Y_0 \ar[d] \\
    X_0 \times Y_1 \ar[r] & (X \gprod Y)_1
\end{tikzcd} \]
For instance, when $X = Y = \cube{1}$, we visualize the non-degenerate 1-cubes of $\cube{1} \gprod \cube{1} \cong \cube{2}$ as a gluing
\[ \begin{tikzcd}
    00 & 10 \ar[d, phantom, ""{name=TLtoTR}] & {} & 00 \ar[r] \ar[d, phantom, ""{name=TRfromTL}] \ar[from=TLtoTR, to=TRfromTL, shorten=2.2em] & 10 \\
    01 \ar[r, phantom, ""{name=TLtoBL}] & 11 & {} & 01 \ar[r, ""{name=TRtoBR}] & 11 \\
    {} & {} & {} & {} & {} \\
    00 \ar[d] \ar[r, phantom, ""{name=BLfromTL}] \ar[from=TLtoBL, to=BLfromTL, shorten=1.7em] & 10 \ar[d, ""{name=BLtoBR}] & {} & 00 \ar[r, ""{name=BRfromTR}] \ar[d, ""{name=BRfromBL}] \ar[from=TRtoBR, to=BRfromTR, shorten=1.7em] \ar[from=BLtoBR, to=BRfromBL, shorten=2.2em] & 10 \ar[d] \\
    01 & 11 & {} & 01 \ar[r] & 11
\end{tikzcd} \]
where the 1-cubes of $X$ appear along the horizontal and the 1-cubes of $Y$ appear along the vertical.

This product is biclosed: for a cubical set $X$, we write $\lhom[L](X, \uvar) \from \cSet \to \cSet$ and $\rhom (X, \uvar) \from \cSet \to \cSet$ for the right adjoints to the functors $\uvar \gprod X$ and $X \gprod \uvar$, respectively.

The geometric product is better behaved than the categorical product from the point of view of homotopy theory.
For instance, the categorical product of cubes is not isomorphic to a cube (in fact, for cubical sets without connections, the categorical product of cubes is not usually contractible).
In this sense, the geometric product more closely resembles the product of topological spaces.

We use the geometric product to define pushout products and pullback exponential maps.
\begin{definition}
    For maps $f \from A \to B$ and $g \from X \to Y$,
    \begin{enumerate}
        \item the \emph{pushout product} $f \pp g$ of $f$ and $g$ is the map $A \gprod Y \cup_{A \gprod X} B \gprod X \to B \gprod Y$ from the pushout induced by the commutative square
        \[ \begin{tikzcd}
            A \gprod X \ar[r, "f \gprod X"] \ar[d, "A \gprod g"'] \ar[rd, phantom, "\ulcorner" very near end] & B \gprod X \ar[d] \ar[ddr, "B \gprod g", bend left] & {} \\
            A \gprod Y \ar[r] \ar[drr, bend right, "f \gprod Y"] & \bullet \ar[rd, dotted, "f \pp g"] \\
            {} && B \gprod Y
        \end{tikzcd} \]
        \item the (right) \emph{pullback exponential} $f \pe g$ of $f$ along $g$ is the map $\rhom(B, X) \to \rhom(A, X) \times_{\rhom(A, Y)} \rhom(B, Y)$ to the pullback induced by the commutative square
        \[ \begin{tikzcd}
            \rhom(B, X) \ar[rrd, bend left, "f^*"] \ar[ddr, bend right, "g_*"'] \ar[rd, dotted, "f \pe g"] & {} & {} \\
            {} & \bullet \ar[r] \ar[d] \ar[rd, phantom, "\ulcorner" very near start, yshift=-1ex] & \rhom(A, X) \ar[d, "g_*"] \\
            {} & \rhom(B, Y) \ar[r, "f^*"] & \rhom(A, Y)
        \end{tikzcd} \]
    \end{enumerate}    
\end{definition}

\begin{example}[cf.~{\cite[Lem.~1.26]{doherty-kapulkin-lindsey-sattler}}]
    For $n \geq 0$,
    \begin{itemize}
        \item given $m \geq 0$, the boundary inclusion $\bd \cube{m+n} \ito \cube{m+n}$ is the pushout product
        \[ (\bd \cube{m} \ito \cube{m}) \pp (\bd \cube{n} \ito \cube{n})\text{;} \]
        \item given $i = 1, \dots, n$ and $\varepsilon = 0, 1$, the open box inclusion $\dfobox \ito \cube{n}$ is the pushout product
        \[ (\bd \cube{i-1} \ito \cube{i-1}) \pp (\{1 - \varepsilon\} \ito \cube{1}) \pp (\bd \cube{n-i} \ito \cube{n-i}) \text{;}\]
        \item the pushout product of any map with $\varnothing \ito \cube{0}$ is itself.
    \end{itemize}
\end{example}

\begin{example}\label{pbexp-examples}
    For a cubical set $X$,
    \begin{itemize}
        \item the pre-composition map $(\face{*}{1,0}, \face{*}{1,1}) \from \rhom(\cube{1}, X) \to X \times X$ is the pullback exponential of $\bd \cube{1} \ito \cube{1}$ along $X \to \cube{0}$;
        \item the pre-composition map $\face{*}{1,0} \from \rhom(\cube{1}, X) \to X$ is the pullback exponential of $\{ 0 \} \ito \cube{1}$ along $X \to \cube{0}$;
        \item the pullback exponential of $\varnothing \ito \cube{0}$ with any map is itself.
    \end{itemize}
\end{example}
\begin{remark}
    There is a notion of left pullback exponential, however we will only work with right pullback exponentials.
\end{remark}

\subsection{Drawing conventions}
Let $X$ be a cubical set.
We depict a map $u \from \cube{1} \to X$ as an arrow
\[ \begin{tikzcd}
    u\face{}{1,0} \ar[r] & u\face{}{1,1}
\end{tikzcd} \]
starting at the $\face{}{1,0}$-face and ending at the $\face{}{1,1}$-face.
For $x \in X$, the degenerate 1-cube $x\degen{}{1} \from \cube{1} \to X$ is depicted by a double line
\[ \begin{tikzcd}
    x \ar[r, equal] & x
\end{tikzcd} \]
with no arrow head.
We depict a map $v \from \cube{2} \to X$ as a square
\[ \begin{tikzcd}
    \bullet \ar[r, "v\face{}{2,0}"] \ar[d, "v\face{}{1,0}"'] \ar[rd, "v" description, phantom] & \bullet \ar[d, "v\face{}{1,1}"] \\
    \bullet \ar[r, "v\face{}{2,1}"'] & \bullet
\end{tikzcd} \]
We omit labels if they are clear from context.
A map $\dfobox[2] \to X$ is depicted as a square without the $\face{}{i,\varepsilon}$-face, e.g.~a map $\obox{2}{2,1} \to X$
is depicted as
\[ \begin{tikzcd}
    \bullet \ar[d] \ar[r] & \bullet \ar[d] \\
    \bullet & \bullet
\end{tikzcd} \]
We depict a map $\cube{3} \to X$ as a cube
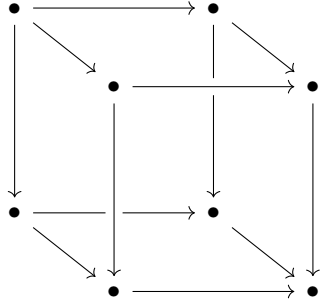
\[ \begin{tikzcd}
    \bullet \ar[rr] \ar[rd] \ar[ddd] && \bullet \ar[rd] \ar[ddd] & \\
    & \bullet \ar[rr, crossing over] && \bullet \ar[ddd] \\
    \\
    \bullet \ar[rr] \ar[rd] && \bullet \ar[rd] & \\
    & \bullet \ar[rr] \ar[from=uuu, crossing over] && \bullet
\end{tikzcd} \]
where
\begin{itemize}
    \item the $\face{}{1,0}$- and $\face{}{1,1}$-faces are the left and right side faces, respectively;
    \item the $\face{}{2,0}$- and $\face{}{2,1}$-faces are the back and front faces, respectively;
    \item the $\face{}{3,0}$- and $\face{}{3,1}$-faces are the top and bottom faces, respectively.
\end{itemize}
We also depict a map $\dfobox[3] \to X$ as a cube.
That is, our diagram will not specify that there is a missing face.

\subsection{Homotopies and homotopy equivalences}
We begin by defining Kan fibrations and Kan complexes.
\begin{definition}\leavevmode
    \begin{enumerate}
        \item A cubical map $f \from X \to Y$ is a \emph{Kan fibration} if it has the right lifting property with respect to open box inclusions.
        That is, given a commutative square,
        \[ \begin{tikzcd}
            \dfobox \ar[r] \ar[d] & X \ar[d, "f"] \\
            \cube{n} \ar[r] & Y
        \end{tikzcd} \]
        there exists a map $\cube{n} \to X$ which makes the triangles in the diagram
        \[ \begin{tikzcd}
            \dfobox \ar[r] \ar[d] & X \ar[d, "f"] \\
            \cube{n} \ar[r] \ar[ur, dotted] & Y
        \end{tikzcd} \]
        commute.
        \item A cubical set $X$ is a \emph{Kan complex} if the  map $X \to \cube 0$ is a Kan fibration.
    \end{enumerate}
\end{definition}
We write $\Kan$ for the full subcategory of $\cSet$ spanned by Kan complexes.

To define a notion of homotopy between cubical maps, we first define elementary homotopies.
\begin{definition}
    Let $f, g \from X \to Y$ be cubical maps. An \emph{elementary homotopy} from $f$ to $g$ is a map $H \from X \gprod \cube{1} \to Y$ such that the diagram
    \[ \begin{tikzcd}
        X \gprod \cube{0} \ar[d, left, "\face{}{1,0}"'] \ar[dr, "f"] & \\
        X \gprod \cube{1} \ar[r, "H"] & Y \\
        X \gprod \cube{0} \ar[u, left, "\face{}{1,1}"] \ar[ur, "g"'] & 
    \end{tikzcd} \]
    commutes.
\end{definition}
By the adjunction $\adjunct{X \gprod \uvar}{\rhom(X, \uvar)}{\cSet}{\cSet}$, an elementary homotopy from $f$ to $g$ corresponds to a 1-cube from $f$ to $g$ in $\rhom(X, Y)$ i.e.~a map $\cube{1} \to \rhom(X, Y)$ whose $\face{}{1,0}$-face is $f$ and whose $\face{}{1,1}$-face is $g$.
\begin{remark} \label{rem:htpy_def}
    One can equivalently define elementary homotopies by tensoring on the left with $\cube{1}$.
    The two definitions agree when considering homotopies between maps into a Kan complex, but differ in general (for instance, $\face{}{2,0}, \face{}{2,1} \from \cube{1} \to \cube{2}$ are homotopic when tensoring with $\cube{1}$ on the right, but not on the left).
    We consistently use the definition given above, and any departures from this convention are remarked on explicitly.
\end{remark}

\begin{proposition} \label{thm:htpy_eq_rel}
    If $Y$ is Kan then the relation of elementary homotopy defines an equivalence relation on maps $X \to Y$. \noproof
\end{proposition}
We omit a proof of this statement and note that a generalization of this result is proven in \cref{thm:rel_htpy_eq_rel}.
In light of \cref{thm:htpy_eq_rel}, we write $[X, Y]$ for the set of homotopy classes of maps $X \to Y$.

If $Y$ is not Kan, the relation of elementary homotopy may not be symmetric or transitive.
We define the notion of homotopy to be the symmetric transitive closure of elementary homotopy.
\begin{definition}
    Let $f, g \from X \to Y$ be cubical maps. A \emph{homotopy} from $f$ to $g$ is a zig-zag of elementary homotopies from $f$ to $g$ (i.e.~a zig-zag of 1-cubes in $\rhom(X, Y)$ from $f$ to $g$).
\end{definition}
\begin{remark}
    As discussed previously, if $Y$ is a Kan complex then the notion of homotopy and elementary homotopy coincide.
\end{remark}
This notion of homotopy allows us to define the notion of \emph{homotopy equivalence}.
\begin{definition}
    For cubical sets $X, Y$, a map $f \from X \to Y$ is a \emph{homotopy equivalence} if there exist
    \begin{itemize}
        \item a map $g \from Y \to X$;
        \item a homotopy from $gf$ to $\id[X]$;
        \item a homotopy from $fg$ to $\id[Y]$.
    \end{itemize}
\end{definition}

\begin{proposition}
    For $n \geq 1$, the map $\degen{}{n} \from \cube{n} \to \cube{n-1}$ is a homotopy equivalence with homotopy inverse given by $\face{}{n,1} \from \cube{n-1} \to \cube{n}$.
\end{proposition}
\begin{proof}
    We have by cubical identites that $\degen{}{n}\face{}{n,1} = \id[\cube{n-1}]$.
    The map $\conn{}{n,0} \from \cube{n} \gprod \cube{1} \to \cube{n}$ exhibits $\face{}{n,1}$ as a homotopy retraction of $\degen{}{n}$ since $\conn{}{n,0}\face{}{n+1,0} = \id[\cube{n}]$ and $\conn{}{n,0}\face{}{n+1,1} = \face{}{n,1}\degen{}{n}$.
\end{proof}
\begin{remark}
    The positive connection map $\conn{}{n,1} \from \cube{n+1} \to \cube{n}$ exhibits $\face{}{n,0}$ as a homotopy retraction of $\degen{}{n}$, thus showing the map $\face{}{n,0}$ is a homotopy equivalence.
\end{remark}
From this, we derive the following corollary.
\begin{corollary} \label{thm:cube_n_contr}
    For $n \geq 0$, the map $\face{}{1,1} \dots \face{}{n,1} \from \cube{0} \to \cube{n}$ is a homotopy equivalence. \noproof
\end{corollary}

\subsection{The model structure on cubical sets}
\begin{theorem}[Cisinski, cf.~{\cite[Thm.~1.34]{doherty-kapulkin-lindsey-sattler}}]
    The category of cubical sets has a model structure where
    \begin{itemize}
        \item cofibrations are monomorphisms;
        \item fibrations are Kan fibrations;
        \item weak equivalences are maps $X \to Y$ which induce bijections $[X, Z] \to [Y, Z]$ for all Kan complexes $Z$. \qed
    \end{itemize}
\end{theorem}
For cubical sets without connections, this is proven by Cisinski \cite[Thm.~8.4.38]{cisinski:presheaves} and Jardine \cite[Ex.~4.23]{jardine:a-beginning}.
When the cube category has one or both connections, the proof is essentially similar: the case of one connection is considered by Cisinski \cite[Thm.~1.7]{cisinski:elegant} and the case of both connections is recorded in \cite[Thm.~1.34]{doherty-kapulkin-lindsey-sattler}.

In particular, the fibrant objects are exactly Kan complexes.
By \cref{ex:model_cat_fib_cat}, the full subcategory $\Kan$ of Kan complexes is a fibration category.
\begin{theorem} \label{thm:kan_fib}
    The category $\Kan$ has a fibration category structure where
    \begin{itemize}
        \item fibrations are maps which have the right lifting property with respect to all open box inclusions;
        \item weak equivalences are maps $X \to Y$ which induce bijections $[Y, Z] \to [X, Z]$ for all Kan complexes $Z$. \noproof
    \end{itemize} 
\end{theorem}
\begin{proposition} \label{thm:pp_anodyne}
    We have that
    \begin{enumerate}
        \item the pushout product $f \pp g$ of cofibrations $f, g$ is a cofibration which is acyclic if either $f$ or $g$ is;
        \item the pullback exponential $f \pe g$ of a cofibration $f$ along a fibration $g$ is a fibration which is acyclic if either $f$ or $g$ is.
    \end{enumerate}
\end{proposition}
\begin{proof}
    The first statement is proven in \cite[Lem.~2.9]{doherty-kapulkin-lindsey-sattler}.
    \cite[Obs.~4.11]{riehl-verity} notes the second statement is equivalent to the first by the adjunction in \cite[Lem.~4.10]{riehl-verity}.
\end{proof}
From this, we derive the following technical corollaries.
\begin{corollary} \label{thm:htpy_bd}
    Let $A \ito B$ be a monomorphism and $X$ be a Kan complex.
    Given a homotopy $H \from A \gprod \cube{1} \to X$ between maps $f, g \from A \to X$, the map $f$ has a lift $\overline{f} \from B \to X$ if and only if $g$ does.
\end{corollary}
\begin{proof}
    A lift $\overline{f} \from B \to X$ of $f$ gives a map $[H, \overline{f}] \from A \gprod \cube{1} \cup_{A \gprod \{ 0 \}} B \gprod \{ 0 \} \to X$.
    By \cref{thm:pp_anodyne}, the pushout product of $A \ito B$ and $\{ 0 \} \to \cube{1}$ is anodyne.
    As $X$ is Kan, $[H, \overline{f}]$ has a lift $B \gprod \cube{1} \to X$.
    \[ \begin{tikzcd}
        A \gprod \cube{1} \push_{A \gprod \{0\}} B \gprod \{ 0 \} \ar[d, hook, "\sim"'] \ar[r, "{[H, \overline{f}]}"] & X \\
        B \gprod \cube{1} \ar[ur, dotted]
    \end{tikzcd} \]
    The $\face{}{1,1}$-face of this lift is a lift of $g$.

    For the converse, we apply this argument to the symmetry homotopy from $g$ to $f$.
\end{proof}
\begin{corollary} \label{thm:htpy_obox_general}
    Let $X$ be a Kan complex, $A \ito B$ be an anodyne map, and $H \from A \gprod \cube{1} \to X$ be a homotopy between maps $f, g \from A \to X$.
    Given $\overline{f}, \overline{g} \from B \to X$ such that
    \begin{enumerate}
        \item $\restr{\overline{f}}{A} = f$
        \item $\restr{\overline{g}}{A} = g$,
    \end{enumerate} 
    there exists a homotopy $\overline{H} \from B \gprod \cube{1} \to X$ from $\overline{f}$ to $\overline{g}$ such that $\restr{H}{A \gprod \cube{1}} = H$.
\end{corollary}
\begin{proof}
    The pushout product $B \gprod \bd \cube{1} \push_{A \gprod \bd \cube{1}} A \gprod \cube{1} \to B \gprod \cube{1}$ of $\bd \cube{1} \ito \cube{1}$ and $A \ito B$ is anodyne by \cref{thm:pp_anodyne}.
    As $X$ is Kan, the map $[[\overline{f}, \overline{g}], H] \from B \gprod \bd \cube{1} \push_{(A \gprod \bd \cube{1})} A \gprod \cube{1} \to X$ has a lift $B \gprod \cube{1} \to X$.
    \[ \begin{tikzcd}
        B \gprod \bd \cube{1} \push_{(A \gprod \bd \cube{1})} A \gprod \cube{1} \ar[r, "{[[\overline{f}, \overline{g}], H]}"] \ar[d, hook, "\sim"'] & X \\
        B \gprod \cube{1} \ar[ur, dotted] & 
    \end{tikzcd} \]
    This gives the desired homotopy $\overline{H} \from B \gprod \cube{1} \to X$.
\end{proof}
\begin{corollary} \label{thm:htpy_obox}
    Let $f \from \dfobox \to X$ be a map into a Kan complex $X$.
    For any two fillers $g, h \from \cube{n} \to X$ of $f$, we have a homotopy $H \from \cube{n-1} \gprod \cube{1} \to X$ from $g\face{}{i,\varepsilon}$ to $h\face{}{i,\varepsilon}$ such that $\restr{H}{\bd \cube{n-1} \gprod \cube{1}} = f\face{}{i,\varepsilon} \gprod \degen{}{1}$.
\end{corollary}
\begin{proof}
    By \cref{thm:htpy_obox_general}, we have a homotopy $\eta \from \cube{n} \gprod \cube{1} \to X$ from $g$ to $h$ which restricts to the reflexivity homotopy $f \gprod \degen{1} \from \dfobox \gprod \cube{1} \to X$.
    The $\face{}{i,\varepsilon}$-face of this homotopy is the desired homotopy.
\end{proof}
\cref{thm:pp_anodyne} also allows us to explicitly describe a factorization of the diagonal $X \to X \times X$ in $\Kan$.
\begin{proposition}
    For a Kan complex $X$, the triangle
    \[ \begin{tikzcd}
        {} & \rhom(\cube{1}, X) \ar[dr, "({\face{*}{1,0}, \face{*}{1,1}})"] & {} \\
        X \ar[rr] \ar[ur, "{\degen{*}{1}}"] && X \times X
    \end{tikzcd} \]
    gives a factorization of the diagonal map $X \to X \times X$ as a weak equivalence followed by a fibration.
\end{proposition}
\begin{proof}
    The commutative diagram of maps
    \[ \begin{tikzcd}
        \cube{0} \ar[d, "{\face{}{1,0}}"'] \ar[rd, "{\id[\cube{0}]}"] \\
        \cube{1} \ar[r, "{\degen{}{1}}"] & \cube{0} \\
        \cube{0} \ar[u, "{\face{}{1,1}}"] \ar[ur, "{\id[\cube{0}]}"']
    \end{tikzcd} \]
    induces a commutative diagram of maps
    \[ \begin{tikzcd}
        {} & X \\
        X \ar[r, "{\degen{*}{1}}" near end] \ar[ur, "{\id[X]}"] \ar[rd, "{\id[X]}"'] & \rhom(\cube{1}, X) \ar[u, "{\face{*}{1,0}}"'] \ar[d, "{\face{*}{1,1}}"] \\
        {} & X
    \end{tikzcd} \]
    The two vertical maps on the right are acyclic fibrations by \cref{thm:pp_anodyne}.
    By two-out-of-three, the middle horizontal map is a weak equivalence.

    We have that the triangle
    \[ \begin{tikzcd}
        {} & \rhom(\cube{1}, X) \ar[dr, "({\face{*}{1,0}, \face{*}{1,1}})"] & {} \\
        X \ar[rr] \ar[ur, "{\degen{*}{1}}", "\sim"'] && X \times X
    \end{tikzcd} \]
    commutes.
    \cref{thm:pp_anodyne,pbexp-examples} show the right map is a fibration, hence giving the desired factorization.
\end{proof}

\subsection{Equivalence with topological spaces}
We show that the model structure on cubical sets is Quillen equivalent to the classical model structure on spaces.

The mapping $[1]^n \mapsto (\simp{1})^n$ gives a functor $\boxcat \to \sSet$.
The left Kan extension of this functor along the Yoneda embedding gives the \emph{triangulation} functor $T \from \cSet \to \sSet$.
\[ \begin{tikzcd}
    \boxcat \ar[r] \ar[d] & \sSet  \\
    \cSet \ar[ur, "T"', dotted] 
\end{tikzcd} \]
This functor has a right adjoint $U \from \sSet \to \cSet$ defined by
\[ (UX)_n := \sSet \left( (\simp{1})^n, X \right). \]

Similary, the topological cubes $[0, 1]^n$ assemble into a functor $\boxcat \to \Top$. 
The left Kan extension of this functor along the Yoneda embedding gives the \emph{cubical geometric realization} functor $\creali{\uvar} \from \cSet \to \Top$.
\[ \begin{tikzcd}
    \boxcat \ar[r] \ar[d] & \Top  \\
    \cSet \ar[ur, "\creali{\uvar}"', dotted] 
\end{tikzcd} \]
Its right adjoint is the \emph{cubical singular complex} functor $\cSing \from \Top \to \cSet$ defined by
\[ (\cSing X)_n := \Top \left( [0, 1]^n, X \right). \]
Note this mirrors the definition of the simplicial geometric realization functor and its right adjoint (see \cref{ex:geometric-realization}).

\begin{lemma} \label{thm:creali_to_sreali}
    For any cubical set $X$, there is a homeomorphism
    \[ \creali{X} \cong \sreali{T X} \]
    natural in $X$.
\end{lemma}
\begin{proof}
    We have that $\sreali{T \uvar} \from \cSet \to \Top$ is a left adjoint as a composite of left adjoints.
    The cubical geometric realization $\creali{\uvar}$ is also a left adjoint.
    Thus, it suffices to show $\creali{[1]^n} \cong \sreali{T[1]^n}$ naturally in $[1]^n \in \cube{}$. 
    Observe that 
    \begin{align*}
        \creali{[1]^n} &= [0, 1]^n \\
        &= \sreali{\simp{1}}^n \\
        &\cong \sreali{(\simp{1})^n} \text{ by \cite[Lem.~3.1.8]{hovey}} \\
        &= \sreali{T[1]^n}.
    \end{align*}
\end{proof}
\begin{corollary} \label{thm:cset_top_equiv}
    The adjunction
    \[ \creali{\uvar} \from \cSet \rightleftarrows \Top : \! \cSing \]
    is a Quillen equivalence.
\end{corollary}
\begin{proof}
    The adjunction
    \[ T \from \cSet \rightleftarrows \sSet : \! U \] 
    is a Quillen equivalence by \cite[Thm.~6.26]{doherty-kapulkin-lindsey-sattler}.
    (Without connections, this result first appears in \cite[Prop.~8.4.30]{cisinski:presheaves}.)
    The adjunction
    \[ \sreali{\uvar} \from \sSet \rightleftarrows \Top : \! \sSing \] 
    is a Quillen equivalence by \cite[Thm.~3.6.7]{hovey}.
    Thus, the composite adjunction $\cSet \rightleftarrows \Top$ is a Quillen equivalence.
    \cref{thm:creali_to_sreali} shows this composite is exactly the desired adjunction
    \[ \creali{\uvar} \from \cSet \rightleftarrows \Top : \! \cSing. \]
\end{proof}
This equivalence ascends to an equivalence between pointed cubical sets and pointed spaces.
\begin{proposition} \label{thm:pcset_ptop_equiv}
    The adjunction
    \[ \creali{\uvar} \from \pt{\cSet} \rightleftarrows \pt{\Top} : \! \cSing \]
    is a Quillen equivalence.
\end{proposition}
\begin{proof}
    Follows from \cite[Prop.~1.3.17]{hovey} and \cref{thm:cset_top_equiv}.   
\end{proof}
From this equivalence, we have that the unit is a homotopy equivalence on pointed Kan complexes.
\begin{corollary} \label{thm:unit_equiv}
    Let $(X,x)$ be a pointed Kan complex.
    The unit map $(X,x) \to (\cSing \creali{X} , x )$ is a pointed homotopy equivalence.
\end{corollary}
\begin{proof}
    Follows from \cref{thm:wequiv_is_htpy_equiv,thm:pcset_ptop_equiv}.
\end{proof}

\subsection{Relative cubical sets} \label{sec:rel-cset}
Let $\Mono{\cSet}$ denote the full subcategory of $\arr{\cSet}$ spanned by monomorphisms.
Explicitly, its objects are monomorphisms $A \ito X$ and a morphism from $A \ito X$ to $B \ito Y$ is a pair of maps $(f, g)$ which form a commutative square of the following form.
\[ \begin{tikzcd}
    A \ar[r, "g"] \ar[d, hook] & B \ar[d, hook] \\
    X \ar[r, "f"] & Y
\end{tikzcd} \]
We refer to such a morphism as a relative cubical map.
Observe $g$ is uniquely determined by $f$: if $h$ also forms a commutative square with $f$ then $g = h$ as the map $B \ito Y$ is monic.
By a slight abuse of notation, we denote an object $A \ito X$ of $\Mono{\cSet}$ by $(X, A)$, supressing the data of the map itself.
With this, we write a relative cubical map as $f \from (X, A) \to (Y, B)$; the bottom map $X \to Y$ is denoted by $f$ and the top map $A \to B$ is denoted by $\restr{f}{A}$.

\begin{remark}
 A reader familiar with the Reedy theory (see for instance \cite{riehl-verity}) may recognize that the category $\Mono{\cSet}$ is simply a subcategory of cofibrant objects in the Reedy model structure on $\arr{\cSet}$.
 We do not emphasize this view, since the Reedy theory plays no role in our development.
\end{remark}

\begin{remark}
    The mapping $X \mapsto (X, \varnothing)$ embeds $\cSet$ as a full subcategory of $\Mono{\cSet}$.
    As such, the definitions and results related to relative cubical sets generalize the analogous results about (absolute) cubical sets.
\end{remark}
We have a corresponding notion of homotopy between relative cubical maps.
\begin{definition}
    Let $f, g \from (X, A) \to (Y, B)$ be relative cubical maps.
    A \emph{relative homotopy} from $f$ to $g$ is a morphism $H \from (X \gprod \cube{1}, A \gprod \cube{1}) \to (Y, B)$ in $\Mono{\cSet}$ such that $H\face{}{1,0} = f$ and $H\face{}{1,1} = g$.
\end{definition}
\begin{proposition} \label{thm:rel_htpy_eq_rel}
    If $Y, B$ are Kan complexes then relative homotopy is an equivalence relation on relative cubical maps $(X, A) \to (Y, B)$.
\end{proposition}
\begin{proof}
    Given a relative cubical map $f \from (X, A) \to (Y, B)$, a homotopy from $f$ to $f$ is given by $f \gprod \degen{}{1} \from X \gprod \cube{1} \to Y$.

    Fix relative cubical maps $f, g \from (X, A) \to (Y, B)$ and a relative homotopy $H \from (X \gprod \cube{1}, A \gprod \cube{1}) \to (Y, B)$ from $f$ to $g$.
    We specify a map $A \gprod \obox{2}{2,1} \to B$ by the following assignment on faces. We may also specify this map by adjointness as a map $\obox{2}{2,1} \to \rhom(A, B)$.
    \[ \begin{array}{l l}
        A \gprod \face{}{1,0} := \restr{H}{A} & A \gprod \face{}{1,1} := \restr{f}{A} \gprod \degen{}{1} \\
        A \gprod \face{}{2,0} := \restr{f}{A} \gprod \degen{}{1}
    \end{array} \qquad \begin{tikzcd}
        \restr{f}{A} \ar[d, "{H}"'] \ar[r, equal] & \restr{f}{A} \ar[d, equal] \\
        \restr{g}{A} & \restr{f}{A}
    \end{tikzcd} \]
    By \cref{thm:pp_anodyne}, we have a lift $K \from A \gprod \cube{2} \to B$.
    \[ \begin{tikzcd}
        A \gprod \obox{2}{2,1} \ar[d] \ar[r] & B \\
        A \gprod \cube{2} \ar[ur, "K"', dotted]
    \end{tikzcd} \]
    The restriction $\restr{K}{A \gprod \face{}{2,1}} \from A \gprod \cube{1} \to B$ of $K$ to the bottom face is a homotopy from $\restr{g}{A}$ to $\restr{f}{A}$.
    We define a map $X \gprod \obox{2}{2,1} \to Y$ by the following faces.
    \[ \begin{array}{l l}
        X \gprod \face{}{1,0} := H & X \gprod \face{}{1,1} := X \gprod \degen{}{1} \\
        X \gprod \face{}{2,0} := X \gprod \degen{}{1}
    \end{array} \qquad \begin{tikzcd}
        f \ar[d, "H"'] \ar[r, equal] & f \ar[d, equal] \\
        g & f
    \end{tikzcd} \]
    From this map and the previous lift, we obtain a map $A \gprod \cube{2} \cup_{A \gprod \obox{2}{2,1}} X \gprod \obox{2}{2,1} \to Y$.
    \cref{thm:pp_anodyne} gives a lift of this map.
    \[ \begin{tikzcd}
        A \gprod \cube{2} \push_{A \gprod \obox{2}{2,1}} X \gprod \obox{2}{2,1} \ar[d, "(A \ito X) \pp (\obox{2}{2,1} \ito \cube{2})"'] \ar[r] & Y \\
        X \gprod \cube{2} \ar[ur, "K'"', dotted]
    \end{tikzcd} \]
    From this, the restriction $\restr{K'}{X \gprod \face{}{2,1}} \from X \gprod \cube{1} \to Y$ of $K'$ to the bottom face is a homotopy from $g$ to $f$ whose restriction to $A \gprod \face{}{2,1}$ is the homotopy $\restr{K}{A \gprod \face{}{2,1}}$ from $\restr{g}{A}$ to $\restr{f}{A}$.
    That is, it is a relative homotopy from $g$ to $f$.
    
    An analogous argument proves this relation is transitive.
\end{proof}
We write $f \rsim g$ if there is a relative homotopy from $f$ to $g$.
The following statement relates homotopies with composition of maps.
\begin{lemma} \label{thm:compose_rel_htpy}
    Suppose we have relative maps $f, g \from (X, A) \to (Y, B)$, $h, k \from (Y, B) \to (Z, C)$ and relative homotopies $H \from f \rsim g$ and $K \from h \rsim k$.
    Then we have homotopies \begin{enumerate}
        \item $hf \rsim hg$;
        \item $hf \rsim kf$.
    \end{enumerate}
    Moreover, if $C$ and $Z$ are Kan then
    \begin{enumerate}
        \item[3.] $hf \sim kg$
    \end{enumerate}
\end{lemma}
\begin{proof}
    For (1), the map $hH \from (X \gprod \cube{1}, A \gprod \cube{1}) \to (Z, C)$ is a relative homotopy from $hf$ to $hg$.
    
    For (2), the map $K \circ f \gprod \cube{1} \from (X \gprod \cube{1}, A \gprod \cube{1}) \to (Z, C)$ is a relative homotopy from $hf$ to $kf$.

    If $C, Z$ are Kan then (3) follows by transitivity.
\end{proof}
\begin{remark}
    This shows there is a well-defined homotopy category of relative Kan complexes whose objects are monomorphisms between Kan complexes and morphisms are relative homotopy classes of maps.
\end{remark}
The notion of relative homotopy naturally gives a notion of relative homotopy equivalence.
\begin{definition}
    A relative map $f \from (X, A) \to (Y, B)$ between Kan complexes is a \emph{relative homotopy equivalence} if there exists
    \begin{itemize}
        \item a relative map $g \from (Y, B) \to (X, A)$;
        \item a relative homotopy $gf \rsim \id[(X, A)]$;
        \item a relative homotopy $fg \rsim \id[(Y, B)]$.
    \end{itemize}
\end{definition}
\begin{proposition}
    Let $f \from (X, A) \to (Y, B)$ be a relative homotopy equivalence.
    For any relative Kan complex $(Z, C)$, we have
    \begin{enumerate}
        \item an isomorphism $f_* \from [(Z, C), (X, A)] \to [(Z, C), (Y, B)]$ 
        \item an isomorphism $f^* \from [(Y, B), (Z, C)] \to [(X, A), (Z, C)]$
    \end{enumerate}
\end{proposition}
\begin{proof}
    This follows from \cref{thm:compose_rel_htpy}.
\end{proof}

\section{Homotopy groups of a cubical Kan complex} \label{sec:groups}

\subsection{Path components of a Kan complex}
Let $\pi_0 \from \cSet \to \Set$ denote the functor given by computing the colimit of $X$ when regarded as a diagram $\boxcat^{\op} \to \Set$. Note that if $X$ is a Kan complex, this gives the set of homotopy classes of maps $\cube{0} \to X$.
It is straightforward to show that this functor takes weak equivalences to bijections.
\begin{proposition} \label{thm:pi_0_we_to_iso}
    The functor $\pi_0 \from \Kan \to \Set$ takes weak equivalences to bijections. \noproof
\end{proposition}

\subsection{Loop space of a Kan complex}
Recall that the homotopy groups of a topological space may be defined as the path components of the $n$-th loop space.
We develop a similar definition below for cubical Kan complexes.

Note that many of our definitions and theorems make sense in a far greater generality than what is presented here.
For instance, the notion of a loop space makes sense and can be meaningfully studied in an arbitrary fibration category, cf.~\cite{cisinski:invariance}.

Define the \emph{right loop space} $\loopr \from \pt{\cSet} \to \pt{\cSet}$ of $(X,x)$ to be the pullback of $(x,x) \from \cube{0} \to X \times X$ along the map $({\face{*}{1,0}}, {\face{*}{1,1}}) \from \rhom(\cube{1},X) \to X \times X$.
\[ \begin{tikzcd}[column sep = large]
    {\loopr (X,x)} \ar[r] \ar[d] \ar[dr, phantom, "\ulcorner" very near start] & \cube{0} \ar[d, "{(x,x)}"] \\
    {\rhom(\cube{1},X)} \ar[r, "{({\face{^*}{1,0}}, {\face{^*}{1,1}})}"] & X \times X
\end{tikzcd} \]
We have a distinguished basepoint $x\degen{}{1} \from \cube{0} \to \loopr(X,x)$, thus $\loopr (X,x)$ is a pointed cubical set.
Given a pointed map $f \from (X,x) \to (Y,y)$, we have the following commutative square.
\[ \begin{tikzcd}[column sep = large]
    \loopr(X,x) \ar[r] \ar[d] & \cube{0} \ar[d, "{(x,x)}"] \\
    {\rhom(\cube{1},X)} \ar[d, "f_*"'] & X \times X \ar[d, "f \times f"] \\
    {\rhom(\cube{1},Y)} \ar[r, "{(\face{*}{1,0}, \face{*}{1,1})}"] & Y \times Y
\end{tikzcd} \]
This induces a map $\loopr f \from \loopr(X,x) \to \loopr(Y,y)$.
We will refer to the right loop space as simply the \emph{loop space} and write $\loopsp(X,x)$ and $\loopsp f$ for its action on objects and morphisms, respectively. 

We identify an $n$-cube $\cube{n} \to \loopsp(X,x)$ of the loop space with the map $\cube{n+1} \to X$ corresponding to the composite $\cube{n} \to \loopsp(X,x) \to \rhom(\cube{1},X)$.
Explicitly,
\begin{itemize}
    \item a 0-cube of $\loopsp (X,x)$ is a 1-cube of $X$ whose faces are both $x$.
    \item a 1-cube of $\loopsp(X,x)$ is a 2-cube of $X$ of the following form.
    \[ \begin{tikzcd}
        x \ar[r] \ar[d, equal] & x \ar[d, equal] \\
        x \ar[r] & x
    \end{tikzcd} \]
    \item an $n$-cube of $\loopsp(X,x)$ is an $(n+1)$-cube of $X$ whose $\face{}{1,0}$- and $\face{}{1,1}$-faces are degenerate at $x$.
\end{itemize}

One may analogously define a \emph{left loop space} $\loopl \from \pt{\cSet} \to \pt{\cSet}$ using the cubical set $\lhom(\cube{1},X)$.
Our choice to work with the right loop space rather than the left is dictated by our convention of tensoring on the right when defining homotopies (\cref{rem:htpy_def}).
More precisely, this definition gives that $n$-cubes of the loop space are homotopies between $(n-1)$-cubes, thus giving the natural map in \cref{thm:loop_space_is_cube_maps}.

\begin{proposition} \label{thm:loopsp_lim}
    The loop space functor $\loopsp \from \pt{\cSet} \to \pt{\cSet}$ preserves limits.
\end{proposition}
\begin{proof}
    Follows since limits commute with limits and $\rhom(\cube{1}, \uvar) \from \cSet \to \cSet$ preserves limits (as it is a right adjoint).
\end{proof}

\begin{lemma} \label{thm:loopsp_fib}
    The loop space functor $\loopsp \from \pt{\cSet} \to \pt{\cSet}$ preserves fibrations and acyclic fibrations.
\end{lemma}
\begin{proof}
    Let $f \from (X,x) \to (Y,y)$ be a fibration and
    \[ \begin{tikzcd}
        \dfobox \ar[r, "u"] \ar[d] & \loopsp (X,x) \ar[d, "\loopsp f"] \\
        \cube{n} \ar[r, "v"] & \loopsp (Y,y)
    \end{tikzcd} \]
    be a commutative square.
    By adjointness, the composite map $\dfobox \xrightarrow{u} \loopsp(X,x) \to \rhom(\cube{1}, X)$ corresponds to a map $\overline{u} \from \cube{1} \gprod \dfobox \to X$.
    By definition of the loop space, the faces of the restrictions $\restr{\overline{u}}{\{ 0 \} \gprod \dfobox}, \restr{\overline{u}}{\{ 1 \} \gprod \dfobox} \from \dfobox \to X$ are degenerate at $x$.
    Thus, we may extend $\overline{u}$ to a map $u' \from \obox{n+1}{i+1,\varepsilon} \to X$ defined by the face assignment
    \[ u'\face{}{j,\delta} := \begin{cases}
        x\degen{1}{1} \dots \degen{n}{1} & j = 1 \\
        \restr{\overline{u}}{\cube{1} \gprod \face{}{j-1,\delta}} & \text{otherwise}.
    \end{cases} \]
    By constuction, this map fits into a commutative square
    \[ \begin{tikzcd}
        \obox{n+1}{i+1,\varepsilon} \ar[r, "u'"] \ar[d] & X \ar[d, "f"] \\
        \cube{n+1} \ar[r, "\overline{v}"] & Y
    \end{tikzcd} \]
    where $\overline{v} \from \cube{n+1} \to Y$ is the map corresponding to the composite $\cube{n} \xrightarrow{v} \loopsp(Y,y) \to \rhom(\cube{1}, Y)$.
    As $f$ is a fibration, this square has a lift $\cube{n+1} \to X$.
    The $\face{}{n+1,0}$- and $\face{}{n+1,1}$-faces of this lift are degenerate at $x$, thus giving a map $\cube{n} \to \loopsp(X,x)$ which lifts the starting square by construction.
    Thus, $\loopsp f$ is a fibration.

    The case of acyclic fibrations follows by an analogous argument.
\end{proof}
\begin{corollary} \label{thm:loopsp_kan}
    The loop space $\loopsp(X,x)$ of a pointed Kan complex $(X,x)$ is Kan. \noproof
\end{corollary}
By \cref{thm:loopsp_kan}, the restriction of the loop space to pointed Kan complexes gives a functor $\loopsp \from \pt{\Kan} \to \pt{\Kan}$.
\cref{thm:loopsp_fib} and Ken Brown's lemma \cite[Lem.~1.1.12]{hovey} show that this restriction preserves weak equivalences.
\begin{corollary} \label{thm:loopsp_we}
    The loop space functor $\loopsp \from \pt{\Kan} \to \pt{\Kan}$ preserves weak equivalences between pointed Kan complexes. \noproof
\end{corollary}

\begin{theorem} \label{thm:loopsp_exact}
    The loop space functor $\loopsp \from \pt{\Kan} \to \pt{\Kan}$ is exact.
\end{theorem}
\begin{proof}
    By \cref{thm:loopsp_fib}, it preserves fibrations and acyclic fibrations.
    From \cref{thm:loopsp_lim}, we have that it preserves pullbacks along fibrations and the terminal object.
\end{proof}

We recall the definition of the double mapping path space of a co-span.

\begin{definition}
    Given cubical maps $f \from X \to Z$ and $g \from Y \to Z$ between Kan complexes, the \emph{double mapping path space} of $f$ and $g$ is the pullback of $f \times g \from X \times Y \to Z \times Z$ along the pre-composition map $(\face{*}{1,0}, \face{*}{1,1}) \from \rhom(\cube{1}, Z) \to Z \times Z$, denoted $P(f, g)$.
    \[ \begin{tikzcd}
        P(f, g) \ar[r] \ar[d] \ar[rd, phantom, "\ulcorner" very near start] & \rhom(\cube{1}, Z) \ar[d, "{(\face{*}{1,0}, \face{*}{1,1})}"] \\
        X \times Y \ar[r, "f \times g"] & Z \times Z
    \end{tikzcd} \]
\end{definition}

The double mapping path space gives an explicit construction of the homotopy pullback of two cubical maps in the sense of \cref{def:weak-hopb}.

\begin{proposition}
    The loop space $\loopsp(X,x)$ of a pointed Kan complex $(X,x)$ is the double mapping path space, and hence the homotopy pullback, of the co-span $(x \from \cube{0} \to X, x \from \cube{0} \to X)$. \qed
\end{proposition}

We show that the loop space functor preserves homotopy pullbacks.

\begin{theorem} \label{thm:loopsp_preserves_hopb}
    The loop space functor $\loopsp \from \pt{\Kan} \to \pt{\Kan}$ preserves homotopy pullbacks (in the sense of \cref{def:weak-hopb}).
\end{theorem}

\begin{proof}
    Suppose
    \[ \begin{tikzcd}
        P \ar[r] \ar[d] & A \ar[d, "f" ,""{name=s, left}] \\
        B \ar[r, "g"', ""{name=t, above}] & C \ar[from=s, to=t, bend right, "H"']
    \end{tikzcd} \]
    is a homotopy pullback.
    The lower square in
    \[ \begin{tikzcd}[sep = small]
        \loopsp P \ar[rrd, bend left] \ar[ddr, bend right] \ar[rd, dotted] & {} & {} \\
        {} & \loopsp P(f, g) \ar[r] \ar[d] & \loopsp (\rpath{C}) \ar[d] \\
        {} & \loopsp (A \times B) \ar[r] & \loopsp (C \times C)
    \end{tikzcd} \]
    is a pullback by \cref{thm:loopsp_lim}.
    The loop space functor preserves products, also by \cref{thm:loopsp_lim}.
    One verifies that $\loopsp (\rpath{C}) \cong \rpath{\loopsp C}$, thus $\loopsp P(f, g) \cong P(\loopsp f, \loopsp g)$ and $\loopsp P \to \loopsp P(f, g) \xrightarrow{\cong} P(\loopsp f, \loopsp g)$ is the canonical map for the following homotopy-commutative square.
    \[ \begin{tikzcd}
        \loopsp P \ar[r] \ar[d] & \loopsp A \ar[d, "\loopsp f" ,""{name=s, left}] \\
        \loopsp B \ar[r, "\loopsp g"', ""{name=t, above}] & \loopsp C \ar[from=s, to=t, bend right, "\loopsp H"']
    \end{tikzcd} \]
    This map is a weak equivalence by \cref{thm:loopsp_we}.
\end{proof}

\subsection{Fundamental group of a Kan complex}
\begin{definition}
    Let $(X,x)$ be a pointed Kan complex.
    The \emph{fundamental group} $\pi_1(X,x)$ of $(X,x)$ is the set of path components of the cubical set $\loopsp (X,x)$.
\end{definition}
We write a representative of $\pi_1(X,x)$ by a 0-cube of $\loopsp(X,x)$; that is, a 1-cube of $X$.

For a pointed Kan complex $(X,x)$, we define a multiplication on $\pi_1(X,x)$ as follows.
Given $f, g \from \cube{0} \to \loopsp(X,x)$, we construct a map $\obox{2}{2,0} \to X$ by the following diagram.
\[ \begin{tikzcd}
    x \ar[d, "f"'] & x \ar[d, equal] \\
    x \ar[r, "g"] & x
\end{tikzcd} \]
A \emph{concatenation square} for $f$ and $g$ is a filler $\langle f, g \rangle \from \cube{2} \to X$ for this open box.
A \emph{concatenation} of $f$ and $g$ is a map $fg \from \cube{1} \to X$ which is the $\face{}{2,0}$-face of some concatenation square for $f$ and $g$.
As $X$ is Kan, a concatenation square for $f$ and $g$ always exists.
\cref{thm:htpy_obox} shows concatenation is unique up to a homotopy which is relative to the boundary.
That is, this gives a well-defined binary operation $\pi_1(X,x) \times \pi_1(X,x) \to \pi_1(X,x)$.
\begin{theorem} \label{pi-1-is-a-group}
    Let $(X,x)$ be a pointed Kan complex.
    The multiplication defined above gives a group structure on $\pi_1(X,x)$.
\end{theorem}
\begin{proof}
    Observe that, given $f \from \cube{0} \to \loopsp(X,x)$, the 2-cubes $f\degen{}{2}$ and $f\conn{}{1,0}$ witness $x\degen{}{1}$ as a left and right unit, respectively.
    \[ \begin{tikzcd}
        x \ar[r, "f"] \ar[d, equal] \ar[dr, phantom, "f\degen{}{2}" description] & x \ar[d, equal] \\
        x \ar[r, "f"'] & x
    \end{tikzcd} \quad \begin{tikzcd}
        x \ar[r, "f"] \ar[d, "f"'] \ar[dr, phantom, "f\conn{}{1,0}" description] & x \ar[d, equal] \\
        x \ar[r, equal] & x
    \end{tikzcd} \]

    For associativity, fix $f, g, h \from \cube{0} \to \loopsp(X,x)$.
    We have a map $\obox{3}{3,0} \to X$ given by the following faces.
    \[ \begin{array}{l l}
        \face{}{1,0} := \cpsq{f, g} & \face{}{1,1} := x\degen{}{1}\degen{}{1} \\
        \face{}{2,0} := \cpsq{f, gh} & \face{}{2,1} := h\degen{}{2} \\
         & \face{}{3,1} := \cpsq{g, h}
    \end{array} \qquad \begin{tikzcd}
        x \ar[rr, "f(gh)"] \ar[dr, "fg" description] \ar[ddd, "f"'] & {} & x \ar[dr, equal] \ar[ddd, equal] & {} \\
        {} & x \ar[rr, crossing over, "h" near start] & {} & x \ar[ddd, equal] \\
        \\
        x \ar[rr, "gh" near start] \ar[dr, "g"'] & {} & x \ar[dr, equal] & {} \\
        {} & x \ar[rr, "h"] \ar[from=uuu, crossing over, equal] & {} & x
    \end{tikzcd} \]
    As $X$ is Kan, we have a filler for this open box.
    The $\face{}{3,0}$-face witnesses $f(gh)$ as the composition of $fg$ and $h$.
    That is, $[f(gh)] = [(fg)h]$.

    Concerning inverses, fix $f \from \cube{0} \to \loopsp(X,x)$ and consider the following map $\obox{2}{1,0} \to X$.
    \[ \begin{tikzcd}
        x \ar[r, equal] & x \ar[d, equal] \\
        x \ar[r, "f"] & x
    \end{tikzcd} \]
    As $X$ is Kan, this map has a filler $H \from \cube{2} \to X$.
    Let $g$ be the $\face{}{2,1}$-face of $H$, so that $H$ witnesses $x\degen{}{1}$ as a composite of $g$ and $f$.
    This shows every element of $\pi_1(X,x)$ has a left inverse.
    In particular, $[g]$ has a left inverse $[h]$.
    We have that
    \[ [f] = [x\degen{}{1}][f] = [hg][f] = [h][gf] = [h][x\degen{}{1}] = [h]. \]
    Thus $[g]$ is a right inverse of $[f]$ as it is a right inverse of $[h]$.
\end{proof}
\begin{remark}
    This proof relies on the existence of negative connections.
    For cubical sets with positive connections, there is an analogous definition of concatenation which also gives a group structure on the fundamental group.
\end{remark}

\subsection{Higher homotopy groups}
We obtain higher homotopy groups by iterating the loop space construction.
\begin{definition}
    Let $(X,x)$ be a pointed Kan complex.
    The \emph{$n$-th loop space} of $X$ is defined inductively by
    \[ \loopsp[n](X,x) := \begin{cases}
        (X,x) & n = 0 \\
        \loopsp (\loopsp[n-1](X,x)) & n > 0.
    \end{cases} \]
\end{definition}
We have the following equality.
\begin{proposition} \label{thm:loopn_flipn}
    For any pointed Kan complex $(X,x)$ and $n \geq 0$, we have
    \[ \loopsp[n+1](X,x) = \loopsp[n](\loopsp (X,x)) \]
\end{proposition}
\begin{proof}
    For $n = 0$, this is clear.

    By induction, suppose this equality holds for some fixed $n \geq 0$.
    Observe that
    \[ \loopsp[n+2](X,x) = \loopsp(\loopsp[n+1](X,x)) = \loopsp (\loopsp[n](\loopsp(X,x))) = \loopsp[n+1](\loopsp(X,x)). \]
\end{proof}

We can now define $\pi_n(X, x) = \pi_0 (\loopsp[n](X,x))$.
As $\pi_0 (\loopsp[n](X,x)) = \pi_1(\loopsp[n-1](X,x))$ and the latter has a group structure, this induces a group structure on all homotopy groups.
In particular, for every $n \geq 1$, we have a functor $\pi_n \from \Kan_* \to \Grp$.

\subsection{Invariance}
In this subsection, we show that higher homotopy groups are invariant under homotopy i.e.~homotopic functions induce identical maps on homotopy groups.
We also show invariance of basepoint, that is, if $x, y \from \cube{0} \to X$ are in the same connected component then $\pi_n(X,x) \cong \pi_n(X,y)$.

Recall that a 0-cube of $\loopsp[n] (X,x)$ is a map $\cube{n} \to X$ which sends the boundary to $x$.
This gives a map 
\[ \Mono{\cSet}((\cube{n}, \bd \cube{n}), (X, x)) \to \loopsp[n](X,x)_0 \]
from the set of relative maps $(\cube{n}, \bd \cube{n}) \to (X,x)$ to the set of 0-cubes of $\loopsp[n](X,x)$.
This map is natural in $(X,x)$.
A relative homotopy between two such relative maps may be regarded as a 1-cube of $\loopsp[n](X,x)$.
This gives a map
\[ \Mono{\cSet}((\cube{n} \gprod \cube{1}, \bd \cube{n} \gprod \cube{1}), (X,x)) \to \loopsp[n](X,x)_1 \] 
which is natural in $(X,x)$.
We show every 0 and 1-cube of $\loopsp[n](X,x)$ arises in this way.
\begin{proposition} \label{thm:loop_space_is_cube_maps}
    Let $(X,x)$ be a pointed Kan complex.
    \begin{enumerate}
        \item The natural map 
        \[ \Mono{\cSet}((\cube{n}, \bd \cube{n}), (X, x)) \to \loopsp[n](X,x)_0 \] 
        is a bijection of sets.
        \item The natural map 
        \[ \Mono{\cSet}((\cube{n} \gprod \cube{1}, \bd \cube{n} \gprod \cube{1}) , (X,x)) \to \loopsp[n](X,x)_1 \] 
        is a bijection of sets.
    \end{enumerate}
\end{proposition}
\begin{proof}
    We proceed by induction on $n$.

    For $n = 1$, this follows by definition of the loop space.

    Fix $n \geq 1$ and suppose the statement holds for this $n$.
    Consider a map $\cube{0} \to \loopsp[n+1](X,x)$.
    By \cref{thm:loopn_flipn}, a map $\cube{0} \to \loopsp[n+1](X,x)$ is a map $\cube{0} \to \loopsp[n](\loopsp(X,x))$.
    By our inductive hypothesis, this corresponds to a map $\cube{n} \to \loopsp(X,x)$ which sends each face to $x\degen{1}{}\dots\degen{n-1}{}$.
    Recall an $n$-cube of $\loopsp(X,x)$ is a map $\cube{n+1} \to X$ which sends the $\face{}{1,0}$ and $\face{}{1,1}$-faces to $x\degen{1}{}\dots\degen{n}{}$.
    As all other faces are also mapped to $x\degen{1}{}\dots\degen{n}{}$, this proves (1).

    Analogously, a map $\cube{1} \to \loopsp[n+1](X,x)$ is a map $\cube{n+1} \to \loopsp(X,x)$ where each face is degenerate at $x\degen{}{1}$ except the $\face{}{n+1,0}$- and $\face{}{n+1,1}$-faces.
    As an $(n+1)$-cube of $\loopsp(X,x)$ is a map $\cube{n+2} \to X$ which sends the $\face{}{1,0}$- and $\face{}{1,1}-$face to $x\degen{1}{}\dots\degen{n+1}{}$, this is exactly a map $\cube{n+2} \to X$ where each face is degenerate at $x$ except the $\face{}{n+2,0}$- and $\face{}{n+2,1}$-faces,
    that is, a morphism in $\Mono{\cSet}$ from $(\cube{n+1} \gprod \cube{1}, \bd \cube{n+1} \gprod \cube{1})$ to $(X,x)$.
\end{proof}
\begin{corollary} \label{thm:pi_n_is_cube_maps}
    Let $(X,x)$ be a pointed Kan complex.
    The natural bijection
    \[ \Mono{\cSet}((\cube{n}, \bd \cube{n}), (X, x)) \cong \loopsp[n](X,x)_0 \] 
    induces a bijection
    \[ [(\cube{n}, \bd \cube{n}), (X,x)] \cong \pi_n(X,x) \]
    natural in $X$. \noproof
\end{corollary}
This gives that $\pi_n \from \pt{\cSet} \to \Grp$ is homotopy invariant.
\begin{theorem} \label{thm:pi_n_htpy_to_eq}
    Let $(X,x), (Y,y)$ be pointed Kan complexes and $H \from X \gprod \cube{1} \to Y$ be a pointed homotopy between pointed maps $f, g \from (X,x) \to (Y,y)$.
    Then $\pi_n f = \pi_n g$.
\end{theorem}
\begin{proof}
    By \cref{thm:pi_n_is_cube_maps}, we view $\pi_n f$ and $\pi_n g$ as functions between relative homotopy classes of relative maps  
    \[ [(\cube{n}, \bd \cube{n}), (X,x)] \to [(\cube{n}, \bd \cube{n}), (Y,y)] \] 
    induced by post-composition.
    This then follows by \cref{thm:compose_rel_htpy}.
\end{proof}
\begin{corollary} \label{thm:pi_n_we_to_iso}
    Let $f \from X \to Y$ be a pointed map between pointed Kan complexes.
    If $f$ is a homotopy equivalence then, for any $x \in X$ and $n \geq 0$, the map $\pi_n f \from \pi_n (X,x) \to \pi_n (Y,f(x))$ is an isomorphism. \noproof
\end{corollary}
We have that if $x, y \from \cube{0} \to X$ are in the same connected component then they induce isomorphic homotopy groups.
\begin{proposition} \label{thm:pi_n_pi_0_iso}
    For a Kan complex $X$, if $x, y \from \cube{0} \to X$ are in the same connected component then there is a homotopy equivalence $f \from (X,x) \to (X,y)$.
\end{proposition}
\begin{proof}
    As $x, y$ are in the same connected component and $X$ is Kan, there is a 1-cube $u \from \cube{1} \to X$ such that $u\face{}{1,0} = x$ and $u\face{}{1,1} = y$.
    This gives a map $[\id[X], u] \from X \gprod \{ 0 \} \cup_{\{ x \} \gprod \{ 0 \}} \{ x \} \gprod \cube{1} \to X$.
    As $\{ 0 \} \to \cube{1}$ is an anodyne map, this map has a lift $H \from X \gprod \cube{1} \to X$.
    The restriction of $H$ to $X \gprod \{ 1 \}$ is a map $f \from X \to X$ which sends $x$ to $y$.
    As $H$ is a homotopy from the identity map to $f$, the map $f$ is a weak equivalence.
    By \cref{thm:wequiv_is_htpy_equiv_ptd}, this map is a pointed homotopy equivalence.
\end{proof}
In particular, \cref{thm:pi_n_we_to_iso} shows $f$ induces an isomorphism $\pi_n f \from \pi_n(X,x) \to (X,y)$ for all $n \geq 0$.

\subsection{Equivalent definitions}
We may also express elements of $\pi_n(X,x)$ as maps from a quotient of the $n$-cube by its boundary into $X$.
For $n \geq 0$, let $\cube{n} / \bd \cube{n}$ be the quotient of $\cube{n}$ by its boundary i.e.~the cubical set given by the following pushout.
\[ \begin{tikzcd}
    \bd \cube{n} \ar[r] \ar[d, hook] \ar[rd, phantom, "\ulcorner" very near end] & \cube{0} \ar[d] \\
    \cube{n} \ar[r] & \cube{n} / \bd \cube{n}
\end{tikzcd} \]
Let $[0] \from \cube{0} \to \cube{n}/\bd \cube{n}$ and $[\id[\cube{n}]] \from \cube{n} \to \cube{n}/\bd\cube{n}$ denote the unique non-degenerate 0- and $n$-cubes of $\cube{n}/\bd \cube{n}$, respectively.
As $\cube{n} / \bd \cube{n}$ is a pushout, we have a bijection
\[ \pt{\cSet}((\cube{n} / \bd \cube{n}, [0]), (X,x)) \cong \Mono{\cSet}((\cube{n}, \bd \cube{n}), (X,x)) \]
natural in $(X,x)$.
By \cref{thm:loop_space_is_cube_maps}, this gives a bijection
\[ \pt{\cSet}((\cube{n} / \bd \cube{n}, [0]), (X,x)) \cong \loopsp[n](X,x)_0 \]
natural in $(X,x)$.
\begin{theorem} \label{thm:pi_n_as_quotient}
    Let $(X,x)$ be a pointed Kan complex and $n \geq 0$.
    The natural bijection
    \[ \pt{\cSet}((\cube{n} / \bd \cube{n}, [0]), (X,x)) \cong \loopsp[n](X,x)_0 \]
    induces a bijection
    \[ [(\cube{n} / \bd \cube{n}, [0]), (X,x)] \cong \pi_n(X,x). \]
\end{theorem}
\begin{proof}
    As $\uvar \gprod \cube{1} \from \cSet \to \cSet$ preserves colimits, the square
    \[ \begin{tikzcd}
        \bd \cube{n} \gprod \cube{1} \ar[r, hook] \ar[d] \ar[rd, phantom, "\ulcorner" very near end] & \cube{n+1} \ar[d] \\
        \cube{1} \ar[r] & (\cube{n} / \bd \cube{n}) \gprod \cube{1}
    \end{tikzcd} \]
    is a pushout.
    Thus, we have a bijection
    \[ \Mono{\cSet}((\cube{n} / \bd \cube{n} \gprod \cube{1}, [0] \gprod \cube{1}), (X,x)) \cong \Mono{\cSet}((\cube{n} \gprod \cube{1}, \bd \cube{n} \gprod \cube{1}), (X,x)) \]
    natural in $(X,x)$.
    Note the left side of this isomorphism is exactly the set of pointed homotopies between pointed maps $(\cube{n} / \bd \cube{n}, [0]) \to (X,x)$.
    The right side of this isomorphism is the set of 1-cubes $\loopsp[n](X,x)_1$ by \cref{thm:loop_space_is_cube_maps}.
\end{proof}

We now move towards showing elements of $\pi_n(X,x)$ correspond to maps $\bd \cube{n+1} \to X$.
We construct a map $\bd \cube{n+1} \to \cube{n} / \bd \cube{n}$ by the following commutative square (using the fact that $\bd \cube{n+1}$ may be written as a pushout).
\[ \begin{tikzcd}
    \bd \cube{n} \gprod \bd \cube{1} \ar[r] \ar[d] \ar[rd, phantom, "\ulcorner" very near end] & \bd \cube{n} \gprod \cube{1} \ar[d] \ar[ddr, "{[0]}", bend left] \\
    \cube{n} \gprod \bd \cube{1} \ar[r] \ar[rrd, "{[[\id[\cube{n}]], [0]]}"', bend right] & \bd \cube{n+1} \ar[dr, dotted] \\
    & & \cube{n} / \bd \cube{n}
\end{tikzcd} \]

Let $0$ denote the map $\face{n+1}{0}\face{n}{0} \dots\face{1}{0} \from \cube{0} \to \bd \cube{n+1}$.
By \cref{thm:pi_n_as_quotient}, pre-composition with the map $\bd \cube{n+1} \to \cube{n} / \bd \cube{n}$ gives a map 
\[ \pi_n (X,x) \to [(\bd \cube{n+1}, 0), (X,x)] \] 
natural in $(X,x)$.
We show that this map is a bijection, so that the $n$-th homotopy group of a cubical set is in bijection with the set of pointed homotopy classes of pointed maps from $\bd \cube{n+1}$.
Moreover, we show a map $\bd \cube{n+1} \to X$ has a filler if and only if it is homotopic to the degenerate $(n+1)$-cube $x\degen{1}{1}\dots\degen{n+1}{1}$.
\begin{lemma} \label{thm:we_bd_to_quotient}
    The map $\bd \cube{n+1} \xrightarrow{\sim} \cube{n}/\bd\cube{n}$ is a weak equivalence.
\end{lemma}
\begin{proof}
    The left and outer squares in the diagram
    \[ \begin{tikzcd}
        \bd \cube{n} \ar[r, "\face{}{n+1,1}"] \ar[d, hook] \ar[rd, phantom, "\ulcorner" very near end] & \obox{n+1}{n+1,1} \ar[r, "\sim"] \ar[d, hook] & \cube{0} \ar[d, hook] \\
        \cube{n} \ar[r, "\face{}{n+1,1}"] & \bd \cube{n+1} \ar[r] & \cube{n} / \bd \cube{n}
    \end{tikzcd} \]
    are pushouts, thus the right square is a pushout.
    The map $\obox{n+1}{n+1,1} \to \cube{0}$ is a weak equivalence as it factors as $\obox{n+1}{n+1,1} \ito \cube{n+1} \to \cube{0}$, where $\obox{n+1}{n+1,1} \ito \cube{n+1}$ is anodyne by definition and $\cube{n+1} \to \cube{0}$ is a weak equivalence by \cref{thm:cube_n_contr}.
    By left properness, the map $\bd \cube{n+1} \to \cube{n} / \bd \cube{n}$ is a weak equivalence.
\end{proof}
\begin{theorem} \label{thm:pi_n_as_bd_maps}
    Let $(X,x)$ be a pointed Kan complex.
    For $n \geq 0$, 
    \begin{enumerate}
        \item the natural map
        \[ \pi_n (X,x) \to [(\bd \cube{n+1}, 0), (X,x)]_* \]
        is a bijection;
        \item for any pointed map $f \from (\bd \cube{n+1}, 0) \to (X,x)$, writing $[\tilde{f}] \in \pi_n (X,x)$ for the preimage of $f$ under this bijection, we have that $[\tilde{f}] = [x\degen{1}{1}\dots\degen{n}{1}]$ if and only if $f$ has a filler $\cube{n} \to X$.
    \end{enumerate}
\end{theorem}
\begin{proof}
    For (1), \cref{thm:we_bd_to_quotient} gives a weak equivalence $\bd \cube{n+1} \xrightarrow{\sim} \cube{n}/\bd\cube{n}$.
    This gives an isomorphism 
    \[ [(\bd \cube{n+1}, 0), (X,x)] \cong [(\cube{n}/\bd\cube{n}, [0]), (X,x)] \]
    by \cref{thm:we_to_hom_iso}.
    Applying \cref{thm:pi_n_as_quotient} gives the desired result.
    
    For (2), if $[\tilde{f}] = [x\degen{1}{1}\dots\degen{n}{1}]$ then $f$ is homotopic to the trivial pointed map $(\bd \cube{n+1}, 0) \to (X, x)$ by \cref{thm:pi_n_as_bd_maps}.
    Thus, $f$ has a filler by \cref{thm:htpy_bd}.

    If $f$ has a filler then $[f]$ lies in the image of the restriction map $[\cube{n+1}, X] \to [\bd \cube{n+1}, X]$.
    By \cref{thm:cube_n_contr,thm:we_to_hom_iso}, the map $[(\cube{n+1}, 0), (X, x)] \to [(\cube{0}, 0), (X,x)]$ is a bijection.
    As $[(\cube{0}, 0), (X,x)] = \{ [x] \}$, we have that the filler for $f$ is homotopic to $x\degen{1}{1}\degen{n+1}{n+1}$, hence $[\tilde{f}] = [x\degen{1}{}\dots\degen{n}{}]$ in $\pi_n (X,x)$.
\end{proof}

Finally, we may extend the definition of homotopy groups to arbitrary cubical sets, rather than just Kan complexes.
\begin{definition}
    Let $n \geq 0$.
    For a cubical set $X$ and $x \from \cube{0} \to X$, we define the \emph{$n$-th homotopy group} of $(X,x)$ to be $\pi_n (RX,x)$, where $RX$ is the fibrant replacement of $X$.
\end{definition}
This is well-defined for any fibrant replacement of $X$, as \cref{thm:fibr_repl_equiv} shows any two fibrant replacements of a cubical set are homotopy equivalent.
By \cref{thm:pi_n_we_to_iso}, their homotopy groups are isomorphic.

\subsection{Comparison to topological spaces}
\begin{theorem} \label{thm:pi_n_sing}
    Let $(X, x)$ be a pointed space.
    We have an isomorphism $\pi_n(X,x) \cong \pi_n(\Sing X, x)$ natural in $(X,x)$.
\end{theorem}
\begin{proof}
    By the geometric realization and cubical complex adjunction, we have a bijection between commutative squares
    \[ \begin{tikzcd}
        \bd \cube{n} \ar[r] \ar[d] & \cSing \{ \ast \} \ar[d, "x"] \\
        \cube{n} \ar[r] & \cSing X
    \end{tikzcd} \leftrightarrow \begin{tikzcd}
        \creali{\bd \cube{n}} \ar[r] \ar[d] & \{ \ast \} \ar[d, "x"] \\
        \creali{\cube{n}} \ar[r] & X
    \end{tikzcd} \]
    natural in $(X,x)$.
    Observe $\cSing \{ \ast \} \cong \cube{0}$.
    As $\creali{\bd \cube{n}} \cong S^n$ and $\creali{\cube{n}} \cong D^n$, this gives a natural bijection between maps $D^n \to X$ which send the boundary to $\{ x \}$ and maps $\cube{n} \to \cSing X$ which send the boundary to $\{ x \}$.
    That is, a natural bijection between pointed maps $S^{n+1} \to X$ and relative maps $(\cube{n}, \bd \cube{n}) \to (\cSing X, x)$.

    Again by adjointness, we have a bijection between commutative squares
    \[ \begin{tikzcd}
        \bd \cube{n} \gprod \cube{1} \ar[r] \ar[d] & \cSing \{ \ast \} \ar[d, "x"] \\
        \cube{n} \gprod \cube{1} \ar[r] & \cSing X
    \end{tikzcd} \leftrightarrow \begin{tikzcd}
        \creali{\bd \cube{n} \gprod \cube{1}} \ar[r] \ar[d] & \{ \ast \} \ar[d, "x"] \\
        \creali{\cube{n} \gprod \cube{1}} \ar[r] & X
    \end{tikzcd} \]
    natural in $(X,x)$.
    By \cref{thm:creali_to_sreali} and \cite[Prop.~1.29]{doherty-kapulkin-lindsey-sattler}, the functor $\creali{\uvar} \from \cSet \to \Top$ takes geometric products to products.
    Thus, we have a natural bijection between relative homotopies of maps $\cube{n} \to \cSing X$ which send the boundary to $\{x\}$ and based homotopies between based maps $S^{n+1} \to X$.
\end{proof}
We have an analogous result for geometric realization.
\begin{theorem}
    Let $(X,x)$ be a pointed Kan complex.
    We have an isomorphism $\pi_n (X, x) \cong \pi_n (\reali{X}, x)$ natural in $(X,x)$.
\end{theorem}
\begin{proof}
    This is a straightforward calculation.
    \[ \begin{array}{r l l}
        \pi_n (X,x) & \cong \pi_n (\Sing \reali{X}, x) & \text{by \cref{thm:unit_equiv,thm:pi_n_we_to_iso}} \\
        & \cong \pi_n (\reali{X}, x) & \text{by \cref{thm:pi_n_sing}}
    \end{array} \]
\end{proof}

\section{Classical Results} \label{sec:results}

\subsection{Preservation of products}
\begin{proposition} \label{thm:pi_n_prod}
    For $n \geq 0$, the homotopy group functor $\pi_n \from \pt{\cSet} \to \Grp$ preserves products.
\end{proposition}
\begin{proof}
    Fix a family of cubical sets $\{ X_i \}_{i \in I}$ and maps $x_i \from \cube{0} \to X_i$.
    We have a group homomorphism $\pi_n (\prod\limits_{i \in I} X_i) \to \prod\limits_{i \in I} \pi_n (X_i, x_i)$ by the universal property of the product.
    From \cref{thm:pi_n_as_quotient}, one deduces this map is a bijection, hence a group isomorphism.
\end{proof}

\subsection{Higher homotopy groups are abelian}
In this subsection, we show that higher homotopy groups of a Kan complex are abelian.
Following the analogous proof for topological spaces, we will prove this by defining two multiplication operations and showing they satisfy the interchange law.
From this, the result follows by the Eckmann-Hilton argument.

For $f, g \from \cube{0} \to \loopsp[2](X,x)$, recall a composition square for $f$ and $g$ is a filler for the map $\obox{3}{3,0} \to X$ defined as follows.
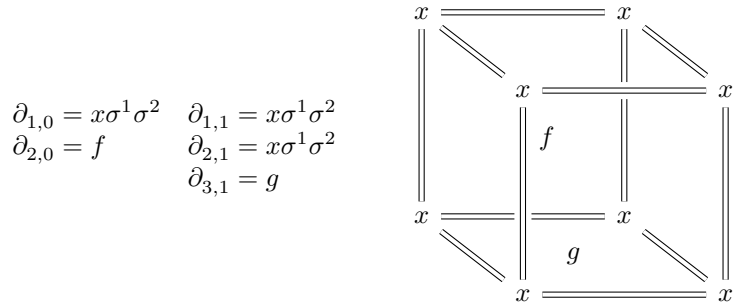
\[ \begin{array}{l l}
    \face{}{1,0} = x\degen{1}{}\degen{2}{} & \face{}{1,1} = x\degen{1}{}\degen{2}{} \\
    \face{}{2,0} = f & \face{}{2,1} = x\degen{1}{}\degen{2}{} \\
    & \face{}{3,1} = g
\end{array} \qquad \begin{tikzcd}[arrows=equal]
    x \ar[rr] \ar[ddd] \ar[rd] && x \ar[rd] \ar[ddd] && \\
    & x \ar[rr, crossing over] && x \ar[ddd] \\
    \\
    x \ar[rr] \ar[rd] \ar[rrrd, phantom, "g" description] && x \ar[rd]  && \\
    & x \ar[from=uuu, crossing over] \ar[from=luuuu, to=ur, phantom, "f" description, yshift=-2ex, xshift=2ex] \ar[rr] && x
\end{tikzcd} \] 
More generally, we may define a composition cube for squares $f, g \from \cube{2} \to X$ such that $f\face{}{2,1} = g\face{}{2,0}$.
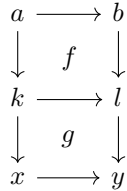
\[ \begin{tikzcd}
    a \ar[r] \ar[d] \ar[rd, phantom, "f" description] & b \ar[d] \\
    k \ar[r] \ar[d] \ar[rd, phantom, "g" description] & l \ar[d] \\
    x \ar[r] & y
\end{tikzcd} \]
Such a composition cube is a filler for the map $\obox{3}{3,0} \to X$ defined by following faces.
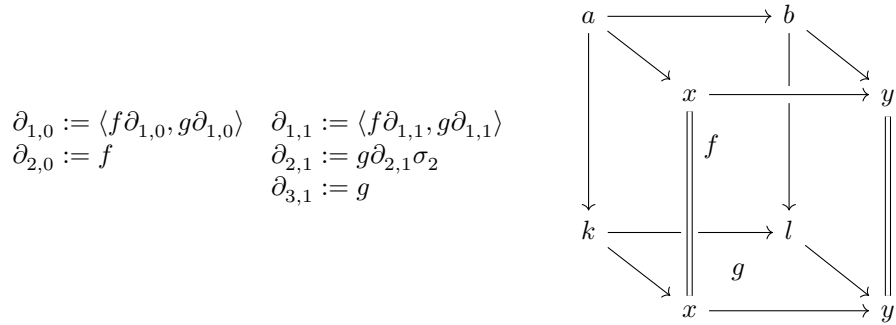
\[ \begin{array}{l l}
    \face{}{1,0} := \cpsq{f\face{}{1,0}, g\face{}{1,0}} & \face{}{1,1} := \cpsq{f\face{}{1,1}, g\face{}{1,1}} \\
    \face{}{2,0} := f & \face{}{2,1} := g\face{}{2,1}\degen{}{2} \\
    & \face{}{3,1} := g
\end{array} \qquad \begin{tikzcd}
    a \ar[rr] \ar[rd] \ar[ddd] && b \ar[rd] \ar[ddd] & \\
    & x \ar[rr, crossing over] && y \ar[ddd, equal] \\
    \\
    k \ar[rr] \ar[rd] \ar[rrrd, phantom, "g" description] && l \ar[rd] & \\
    & x \ar[rr] \ar[from=uuu, equal, crossing over] \ar[from=uuuul, to=ur, phantom, "f" description, yshift=-2ex, xshift=2ex] && y
\end{tikzcd} \] 
Here, $\cpsq{f\face{}{1,0}, g\face{}{1,0}}$ and $\cpsq{f\face{}{1,1}, g\face{}{1,1}}$ are any fillers for the maps $\obox{2}{2,0} \rightrightarrows X$ depicted below.
\[ \begin{tikzcd}
    a \ar[d, "f\face{}{1,0}"'] & x \ar[d, equal] \\
    k \ar[r, "g\face{}{1,0}"] & x
\end{tikzcd} \qquad \begin{tikzcd}
    b \ar[d, "f\face{}{1,1}"'] & y \ar[d, equal] \\
    l \ar[r, "g\face{}{1,1}"] & y
\end{tikzcd} \]
\cref{thm:htpy_obox} shows this composition of $f$ and $g$ is well-defined up to homotopy.
In the case where the boundary of $f$ and $g$ are degenerate at $\{ x \}$, this is exactly the multiplication in $\loopsp[2](X,x)$.
There is also a composition for maps $f, h \from \cube{2} \to X$ such that $f\face{}{1,1} = h\face{}{1,0}$.
\[ \begin{tikzcd}
    a \ar[r] \ar[d] \ar[rd, phantom, "f" description] & b \ar[r] \ar[d] \ar[rd, phantom, "h" description] & z \ar[d] \\
    k \ar[r] & l \ar[r] & w
\end{tikzcd} \]
This composition is given via a filler for the map $\obox{3}{2,0} \to X$ defined by the following faces.
\[ \begin{array}{l l}
    \face{}{1,0} := f & \face{}{1,1} := h\face{}{1,1}\degen{}{1} \\
    & \face{}{2,1} := h \\
    \face{}{3,0} = \cpsq{f\face{}{2,0}, h\face{}{2,0}} & \face{}{3,1} = \cpsq{f\face{}{2,1}, h\face{}{2,1}}
\end{array} \qquad \begin{tikzcd}
    a \ar[rr] \ar[rd] \ar[ddd] \ar[rdddd, phantom, "f" description] && z \ar[rd, equal] \ar[ddd] & \\
    & b \ar[rr, crossing over] \ar[rrddd, "h" description, phantom, xshift=-3ex, yshift=1ex] && z \ar[ddd] \\
    \\
    k \ar[rr] \ar[rd] && w \ar[rd, equal] & \\
    & l \ar[rr] \ar[from=uuu, crossing over] && w
\end{tikzcd} \]
Here, $\cpsq{f\face{}{2,0}, h\face{}{2,0}}$ and $\cpsq{f\face{}{2,1}, h\face{}{2,1}}$ are any squares which witness composition, as before.
Let $[f] \hcirc [h]$ denote (the homotopy class of) this horizontal composition and $[f] \vcirc [g]$ denote the previous vertical composition.
We show that these compositions satisfy the interchange law.
\begin{proposition} \label{thm:interchange}
    Let $X$ be a Kan complex with the following diagram of cubes.
    \[ \begin{tikzcd}[sep = large]
        \cdot \ar[r, "a_1"] \ar[d, "b_1"'] \ar[rd, phantom, "f_1" description] & \cdot \ar[r, "a_2"] \ar[d] \ar[rd, phantom, "f_2" description] & \cdot \ar[d, "b_3"] \\
        \cdot \ar[r, "a_3" description] \ar[d, "b_2"'] \ar[rd, phantom, "g_1" description] & \cdot \ar[r, "a_4" description] \ar[d] \ar[rd, phantom, "g_2" description] & \cdot \ar[d, "b_4"] \\
        \cdot \ar[r, "a_5"'] & \cdot \ar[r, "a_6"'] & \cdot
    \end{tikzcd} \]
    We have that
    \[ ([g_2] \vcirc [f_2]) \hcirc ([g_1] \vcirc [f_1]) = ([g_2] \hcirc [g_1]) \vcirc ([f_2] \hcirc [f_1]). \]
\end{proposition}
\begin{proof}
    We fix composition squares $\cube{2} \to X$ for each of the three horizontal and three vertical compositions of 1-cubes in the diagram.
    This lets us define a map $\obox{4}{4,0} \to X$ by the following faces.
    \[ \begin{array}{l l}
        \face{}{1,0} := \cpsq{f_1 \vcirc g_1} & \face{}{1,1} := \cpsq{ b_3 b_4}\degen{}{1} \\
        \face{}{2,0} := \cpsq{(f_1 \hcirc f_2) \vcirc (g_1 \hcirc g_2)} & \face{}{2,1} := \cpsq{f_2 \vcirc g_2} \\
        \face{}{3,0} := \cpsq{f_1 \hcirc f_2} & \face{}{3,1} := \cpsq{a_5 a_6}\degen{}{3} \\
        & \face{}{4,1} := \cpsq{g_1 \hcirc g_2} \\
    \end{array} \]
    One verifies that this is a valid map (i.e.~that $\face{}{j,\varepsilon'}\face{}{i,\varepsilon} = \face{}{i+1,\varepsilon}\face{}{j,\varepsilon'}$ for $1 \leq j \leq i \leq 4$ and $\varepsilon, \varepsilon' = 0, 1$).
    This map has a filler $\cube{4} \to X$ as $X$ is Kan.
    The $\face{}{4,0}$-face of this map witnesses $(f_1 \hcirc f_2) \vcirc (g_1 \hcirc g_2)$ as a horizontal composition of $f_1 \vcirc g_1$ and $f_2 \vcirc g_2$, as desired. 
\end{proof}
The standard Eckmann-Hilton argument shows that higher homotopy groups are abelian.
\begin{corollary}
    Let $(X,x)$ be a pointed Kan complex.
    We have that $\pi_n(X,x)$ is abelian for $n \geq 2$. \noproof
\end{corollary}

    

\subsection{Long exact sequence of a fibration}
As in the case of topological spaces, fibrations give rise to a long exact sequence of homotopy groups.
Our presentation largely follows \cite[Lem.~32]{mather:pullbacks} and \cite[Thm.~3.8.12]{cisinski:higher-categories} as it relies on taking iterated homotopy pullbacks.

We observe the following lemma regarding homotopy pullbacks, as defined in \cref{def:weak-hopb}.
\begin{theorem} \label{thm:hopb_pi_0}
    The functor $\pi_0 \from \pt\Kan \to \pt\Set$ takes homotopy pullbacks to (strict) pullbacks.
\end{theorem}
\begin{proof}
    It is straightforward to verify $\pi_0$ preserves pullbacks along fibrations.
    The result then follows by \cref{thm:hopb-equiv,thm:pi_0_we_to_iso}.
\end{proof}

With this, homotopy fibers induce long exact sequences of homotopy groups.
\begin{theorem} \label{thm:pi_n_htpy_seq}
    A homotopy pullback
    \[ \begin{tikzcd}
        A \ar[r, "i"] \ar[d] \ar[rd, phantom, "\ulcorner" very near start] & X \ar[d, "f", ""{left, name=s}] \\
        \cube{0} \ar[r, "y"', ""{left, name=t}] & Y \ar[from=s, to=t, "H"', bend right]
    \end{tikzcd} \]
    of pointed Kan complexes induces a long exact sequence
    \[ \begin{tikzcd}
        \dots \rar & \pi_n (A,a) \ar[r, "\pi_n i"] & \pi_n (X,x) \ar[r, "\pi_n f"] \ar[d, phantom, ""{coordinate, name=a}] & \pi_n (Y,y) \ar[dll, rounded corners, 
        to path={ -- ([xshift=2ex]\tikztostart.east)
                  |- (a) [near end]\tikztonodes
                  -| ([xshift=-2ex]\tikztotarget.west)
                  -- (\tikztotarget) }] & {} \\
        & \pi_{n-1} (A,a) \ar[r, "\pi_{n-1} i"] & \pi_{n-1} (X,x) \ar[r, "\pi_{n-1} f"] \ar[d, phantom, "\dots"{description, name=b}] & \pi_{n-1} (Y,y) \ar[dll, rounded corners, 
        to path={ -- ([xshift=2ex]\tikztostart.east)
                |- (b) [near end]\tikztonodes }] & \\
        & \pi_1 (A,a) \ar[r, "\pi_1 i"] \ar[from=b, rounded corners, to path={ -| ([xshift=-2ex]\tikztotarget.west)
                -- (\tikztotarget) }] & \pi_1 (X,x) \ar[d, phantom, ""{coordinate, name=c}] \ar[r, "\pi_1 f"] & \pi_1 (Y,y) \ar[dll, rounded corners, 
        to path={ -- ([xshift=2ex]\tikztostart.east)
                |- (c) [near end]\tikztonodes
                -| ([xshift=-2ex]\tikztotarget.west)
                -- (\tikztotarget) }] & {} \\
        & \pi_0 A \ar[r, "\pi_0 i"] & \pi_0 X \ar[r, "\pi_0 f"] & \pi_0 Y
    \end{tikzcd} \]
\end{theorem}
\begin{proof}
    In the diagram,
    \[ \begin{tikzcd}
        \loopsp(X,x) \ar[r, "\loopsp f"] \ar[d] & \loopsp(Y,y) \ar[r] \ar[d, "j", ""{left, name=s1}] & \cube{0} \ar[d] \\
        \cube{0} \ar[r, "{(x, y, fx\degen{}{1})}"', ""{left, name=t1}] & P(f, y) \ar[from=s1, to=t1, bend right, "G"', shorten=1ex] \ar[r, "p_1"] \ar[d] & X \ar[d, "f", ""{left, name=s2}] \\
        {} & \cube{0} \ar[r, "y"', ""{left, name=t2}] & Y \ar[from=s2, to=t2, bend right, "\alpha"', shorten=1ex]
    \end{tikzcd} \tag{$\ast$} \]
    we define the map $j \from \loopsp(Y,y) \to P(f, y)$ and the homotopies $\alpha \from P(f, y) \to \rpath{Y}$ and $G \from \loopsp(X,x) \to \rhom(\cube{1}, P(f, y))$ as follows:
    the map $j$ is induced by the commutative square
    \[ \begin{tikzcd}
        \loopsp(Y,y) \ar[rrd, bend left] \ar[rdd, bend right, "x"'] \ar[rd, dotted, "j"] & {} & {} \\
        {} & P(f, y) \ar[r] \ar[d, "p_1"'] \ar[rd, phantom, "\ulcorner" very near start] & \rhom(\cube{1}, Y) \ar[d, "{(\face{*}{1,0}, \face{*}{1,1})}"] \\
        {} & X \ar[r, "{(f, y)}"] & Y \times Y
    \end{tikzcd} \]
    whereas the map $\alpha$ is the canonical map $P(f, y) \to \rpath{Y}$ and $G$ is induced by the commutative square,
    \[ \begin{tikzcd}
        \loopsp(X,x) \ar[r] \ar[rdd, bend right, ""'] \ar[rd, dotted, "G"] & \rhom(\cube{1}, X) \ar[rd, "f_* \conn{*}{1,0}"] & {} \\
        {} & \rhom(\cube{1}, P(f, y)) \ar[r] \ar[d, "(p_1)_*"'] \ar[rd, phantom, "\ulcorner" very near start] & \rhom(\cube{2}, Y) \ar[d, "{(\face{*}{1,0}, \face{*}{1,1})}"] \\
        {} & \rhom(\cube{1}, X) \ar[r, "{(f_*, y_*}"] & \rhom(\cube{1}, Y) \times \rhom(\cube{1}, Y)
    \end{tikzcd} \]
    where the lower square is a pullback since $\rpath{\uvar}$ preserves limits.

    A composite of the right two squares in $(\ast)$ is given by
    \[ \begin{tikzcd}
        \loopsp(Y,y) \ar[r] \ar[d] & \cube{0} \ar[d, "y", ""{left, name=s}] \\
        \cube{0} \ar[r, "y"', ""{left, name=t}] & Y \ar[from=s, to=t, bend right, "\alpha j"']
    \end{tikzcd} \]
    By construction, $\alpha j$ is the canonical homotopy making $\loopsp(Y,y) = P(y, y)$ a homotopy pullback.
    As the bottom square in $(\ast)$ is a homotopy pullback by assumption, the top-right square is a homotopy pullback by \cref{thm:htpy_two_pb}.

    A composite of the top two squares in $(\ast)$ is given by
    \[ \begin{tikzcd}
        \loopsp(X,x) \ar[r] \ar[d] & \cube{0} \ar[d, "x", ""{name=s, left}] \\
        \cube{0} \ar[r, "x"', ""{left, name=t}] & X \ar[from=s, to=t, "(p_1)_* G"', bend right]
    \end{tikzcd} \]
    By construction, $(p_1)_* G$ is the canonical homotopy making $\loopsp(X,x)$ a homotopy pullback, thus the top-left square in $(\ast)$ is a homotopy pullback by \cref{thm:htpy_two_pb}.
    That is, each square in $(\ast)$ is a homotopy pullback.

    By \cref{thm:loopsp_preserves_hopb}, the square
    \[ \begin{tikzcd}
        \loopsp (A, a) \ar[r, "\loopsp i"] \ar[d] \ar[rd, phantom, "\ulcorner" very near start] & \loopsp (X,x) \ar[d, "f", ""{left, name=s}] \\
        \cube{0} \ar[r, "y"', ""{left, name=t}] & \loopsp(Y, y) \ar[from=s, to=t, bend right, "\loopsp H"']
    \end{tikzcd} \]
    is a homotopy pullback.

    As $\pi_0$ sends weak equivalences to bijections by \cref{thm:pi_0_we_to_iso}, this gives that each square in
    \[ \begin{tikzcd}
        \pi_1(A, a) \ar[r] \ar[d, "\pi_1 i"'] & \{ \ast \} \ar[d] \\
        \pi_1(X, x) \ar[r, "\pi_1 f"] \ar[d] & \pi_1(Y, y) \ar[r] \ar[d] & \{ \ast \} \ar[d] \\
        \{ \ast \} \ar[r] & \pi_0(A, a) \ar[r, "\pi_0 i"] \ar[d] & \pi_0(X, x) \ar[d, "\pi_0 f"] \\
        {} & \{ \ast \} \ar[r] & \pi_0(Y, y)
    \end{tikzcd} \]
    is a pullback of pointed sets by \cref{thm:hopb_pi_0}.
    This gives an exact sequence,
    \[ \begin{tikzcd}
        \pi_1(A, a) \ar[r, "\pi_1 i"] & \pi_1(X, x) \ar[r, "\pi_1 f"] & \pi_1(Y, y) \ar[r] & \pi_0(A, a) \ar[r, "\pi_0 i"] & \pi_0(X, x) \ar[r, "\pi_0 f"] & \pi_0(Y, y)
    \end{tikzcd} \]
    from which the result follows by induction.
\end{proof}
\begin{corollary} \label{thm:pi_n_seq}
    Let $f \from (X,x) \to (Y,y)$ be a map between pointed Kan complexes and $(A, a)$ be the fiber of $f$.
    \[ \begin{tikzcd}
        A \ar[r, "i"] \ar[d] \ar[rd, phantom, "\ulcorner" very near start] & X \ar[d, "f"] \\
        \cube{0} \ar[r] & Y
    \end{tikzcd} \]
    If $f$ is a fibration then there is a long exact sequence
    \[ \begin{tikzcd}
        \dots \rar & \pi_n (A,a) \ar[r, "\pi_n i"] & \pi_n (X,x) \ar[r, "\pi_n f"] \ar[d, phantom, ""{coordinate, name=a}] & \pi_n (Y,y) \ar[dll, rounded corners, 
        to path={ -- ([xshift=2ex]\tikztostart.east)
                  |- (a) [near end]\tikztonodes
                  -| ([xshift=-2ex]\tikztotarget.west)
                  -- (\tikztotarget) }] & {} \\
        & \pi_{n-1} (A,a) \ar[r, "\pi_{n-1} i"] & \pi_{n-1} (X,x) \ar[r, "\pi_{n-1} f"] \ar[d, phantom, "\dots"{description, name=b}] & \pi_{n-1} (Y,y) \ar[dll, rounded corners, 
        to path={ -- ([xshift=2ex]\tikztostart.east)
                |- (b) [near end]\tikztonodes }] & \\
        & \pi_1 (A,a) \ar[r, "\pi_1 i"] \ar[from=b, rounded corners, to path={ -| ([xshift=-2ex]\tikztotarget.west)
                -- (\tikztotarget) }] & \pi_1 (X,x) \ar[d, phantom, ""{coordinate, name=c}] \ar[r, "\pi_1 f"] & \pi_1 (Y,y) \ar[dll, rounded corners, 
        to path={ -- ([xshift=2ex]\tikztostart.east)
                |- (c) [near end]\tikztonodes
                -| ([xshift=-2ex]\tikztotarget.west)
                -- (\tikztotarget) }] & {} \\
        & \pi_0 A \ar[r, "\pi_0 i"] & \pi_0 X \ar[r, "\pi_0 f"] & \pi_0 Y
    \end{tikzcd} \] 
\end{corollary}
\begin{proof}
    The pullback square $(A,a)$ is a homotopy pullback with the constant homotopy by \cref{thm:hopb-def-equiv}.
\end{proof}

\subsection{Whitehead's Theorem}
In this subsection, we prove the following theorem.
\begin{restatable}{theorem}{whitehead} \label{thm:whitehead}
    Let $f \from X \to Y$ be a map between Kan complexes.
    We have that $f$ is a homotopy equivalence if and only if, for all $x \from \cube{0} \to X$ and $n \geq 0$, the map $\pi_n f \from \pi_n(X,x) \to \pi_n(Y, f(x))$ is an isomorphism.
\end{restatable}
We begin by observing the following lemma about Kan fibrations.
\begin{proposition} \label{thm:we_contr_fib}
    Let $f \from X \to Y$ be a fibration.
    If, for all $y \from \cube{0} \to Y$, the fiber of $y$ is contractible then $f$ is a weak equivalence.
\end{proposition}
\begin{proof}
    We show that $f$ has the right lifting property with respect to boundary inclusions.
    Fix a commutative square
    \[ \begin{tikzcd}
        \bd \cube{n} \ar[r, "u"] \ar[d] & X \ar[d, "f", two heads] \\
        \cube{n} \ar[r, "v"] & Y
    \end{tikzcd} \]
    and let $y_0$ denote the map $v\face{}{n,1}\face{}{n-1,1}\dots\face{}{1,1} \from \cube{0} \to Y$.
    
    As $Y$ is a Kan complex, \cref{thm:cube_n_contr} gives an elementary homotopy $w \from \cube{n+1} \to Y$ from $v \from \cube{n} \to Y$ to the composite $\cube{n} \to \cube{0} \xrightarrow{y_0} Y$.
    Thus, we have the following commutative square.
    \[ \begin{tikzcd}[column sep = small]
        \bd \cube{n} \gprod \{ 0 \} \ar[d] \ar[rr, "u"] && X \ar[d, "f", two heads] \\
        \bd \cube{n} \gprod \cube{1} \ar[r, hook] & \cube{n} \gprod \cube{1} \ar[r, "w"] & Y
    \end{tikzcd} \]
    As $f$ is a fibration, it has lifts against anodyne maps.
    The left leg of this square $\bd \cube{n} \gprod \{ 0 \} \ito \bd \cube{n} \gprod \cube{1}$ is anodyne by \cref{thm:pp_anodyne}.
    Let $G \from \bd \cube{n} \gprod \cube{1} \to X$ be the lift of this square and $u' \from \bd \cube{n} \to X$ be the restriction of $G$ to the endpoint $\bd \cube{n} \gprod \{ 1 \}$.

    Let $A$ be the fiber of $y_0$ under $f$.
    \[ \begin{tikzcd}
        A \ar[r] \ar[d, two heads, "\sim"'] \ar[rd, phantom, "\ulcorner" very near start] & X \ar[d, "f", two heads] \\
        \cube{0} \ar[r, "y_0"] & Y
    \end{tikzcd} \]
    The map $A \overset{\sim}{\twoheadrightarrow} \cube{0}$ is a fibration as it is a pullback of a fibration.
    It is a weak equivalence by assumption, hence an acyclic fibration.
    By construction, $u' \from \bd \cube{n} \to X$ factors through the map $A \to X$.
    \[ \begin{tikzcd}
        \bd \cube{n} \ar[rdd, bend right] \ar[rrd, "u'", bend left] \ar[rd, dotted] \\
        & A \ar[r] \ar[d, two heads, "\sim"'] & X \ar[d, "f", two heads] \\
        & \cube{0} \ar[r, "y_0"] & Y
    \end{tikzcd} \]
    This gives the following commutative square.
    \[ \begin{tikzcd}
        \bd \cube{n} \ar[r] \ar[d, hook] & A \ar[d, two heads, "\sim"] \\
        \cube{n} \ar[r] & \cube{0}
    \end{tikzcd} \]
    As the left map is a monomorphism and the right map is an acyclic fibration, this square has a lift.
    Let $h \from \cube{n} \to X$ be the composite of this lift $\cube{n} \to A$ with the map $A \to X$.

    We now have the following commutative square.
    \[ \begin{tikzcd}
        \cube{n} \gprod \{ 1 \} \cup_{\bd \cube{n} \gprod \{ 1 \}} \bd \cube{n} \gprod \cube{1} \ar[d] \ar[r, "{[h, G]}"] & X \ar[d, "f", two heads] \\
        \cube{n} \gprod \cube{1} \ar[r, "w"] & Y
    \end{tikzcd} \]
    The left map is anodyne by \cref{thm:pp_anodyne}, thus this square has a lift $\cube{n} \gprod \cube{1} \to X$.
    The restriction of this lift to $\cube{n} \gprod \{ 0 \}$ is exactly a lift of the starting square.
\end{proof}

As well, we observe the following lemma regarding lifting a point against a fibration.
\begin{lemma} \label{thm:pi_0_iso_fiber}
    Let $f \from X \to Y$ be a fibration between Kan complexes.
    If $\pi_0 f \from \pi_0 X \to \pi_0 Y$ is a bijection then, for any $y \from \cube{0} \to Y$, there exists $x \from \cube{0} \to X$ such that $fx = y$.
\end{lemma}
\begin{proof}
    Fix $y \from \cube{0} \to Y$.
    We have a 0-cube $x' \from \cube{0} \to X$ of $X$ such that $[fx'] = [y]$ since $\pi_0 f$ is a bijection.
    This gives a 1-cube $u \from \cube{1} \to Y$ of $Y$ such that $u\face{}{1,0} = fx'$ and $u\face{}{1,1} = y$ as $Y$ is Kan.
    As $f \from X \to Y$ is a fibration, we have a lift $\cube{1} \to X$ for the following square.
    \[ \begin{tikzcd}
        \cube{0} \ar[d, hook] \ar[r, "x'"] & X \ar[d, two heads, "f"] \\
        \cube{1} \ar[r, "u"] \ar[ur, dotted] & Y
    \end{tikzcd} \]
    Let $x$ denote the $\face{}{1,1}$-face of this lift.
    By construction, $fx = y$.
\end{proof}

With this, we may prove Whitehead's theorem for Kan complexes.
\begin{proof}[Proof of \cref{thm:whitehead}]
    We show the converse, as the forward direction is proven in \cref{thm:pi_n_we_to_iso}.

    Without loss of generality, we may assume $f$ is a fibration.
    This is because, given an arbitrary map $g$, we may factor it as a trivial cofibration $i$ followed by a fibration $f$.
    The maps $g$ and $i$ induce isomorphisms on all homotopy groups (by assumption and by \cref{thm:pi_n_we_to_iso}, respectively), thus $f$ does by two-out-of-three.
    To show $g$ is a homotopy equivalence, it suffices to show $f$ is.

    Fix $y \from \cube{0} \to Y$.
    By \cref{thm:pi_0_iso_fiber}, we have $x \from \cube{0} \to X$ so that $fx = y$.
    Thus, $f$ is a pointed map $(X,x) \to (Y,y)$.
    Let $(A,a)$ be the fiber of $y$.
    \[ \begin{tikzcd}
        A \ar[r] \ar[d, two heads] \ar[rd, phantom, "\ulcorner" very near start] & X \ar[d, "f", two heads] \\
        \cube{0} \ar[r, "y"] & Y
    \end{tikzcd} \]
    By \cref{thm:pi_n_seq}, we have the following long exact sequence.
    \[ \begin{tikzcd}
        \dots \rar & \pi_n (A,a) \rar & \pi_n (X,x) \ar[r, "\pi_n f", "\cong"'] \ar[d, phantom, ""{coordinate, name=a}] & \pi_n (Y,y) \ar[dll, rounded corners, 
        to path={ -- ([xshift=2ex]\tikztostart.east)
                  |- (a) [near end]\tikztonodes
                  -| ([xshift=-2ex]\tikztotarget.west)
                  -- (\tikztotarget) }] & {} \\
        & \pi_{n-1} (A,a) \rar & \pi_{n-1} (X,x) \ar[r, "\pi_{n-1} f", "\cong"'] \ar[d, phantom, "\dots"{description, name=b}] & \pi_{n-1} (Y,y) \ar[dll, rounded corners, 
        to path={ -- ([xshift=2ex]\tikztostart.east)
                  |- (b) [near end]\tikztonodes }] & \\
        & \pi_1 (A,a) \rar \ar[from=b, rounded corners, to path={ -| ([xshift=-2ex]\tikztotarget.west)
               -- (\tikztotarget) }] & \pi_1 (X,x) \ar[d, phantom, ""{coordinate, name=c}] \ar[r, "\pi_1 f", "\cong"'] & \pi_1 (Y,y) \ar[dll, rounded corners, 
        to path={ -- ([xshift=2ex]\tikztostart.east)
                  |- (c) [near end]\tikztonodes
                  -| ([xshift=-2ex]\tikztotarget.west)
                  -- (\tikztotarget) }] & {} \\
        & \pi_0 A \rar & \pi_0 X \ar[r, "\pi_0 f", "\cong"'] & \pi_0 Y
    \end{tikzcd} \]
    From this, we get that $\pi_n (A,a)$ is trivial for all $n \geq 0$.
    By \cref{thm:pi_n_as_bd_maps}, $A$ has fillers for all boundary maps.
    That is, $A \to \cube{0}$ is a trivial fibration.

    By \cref{thm:we_contr_fib}, we conclude that $f$ is a weak equivalence.
    As $X$ and $Y$ are Kan complexes, $f$ is a homotopy equivalence.
\end{proof}

\appendix
\renewcommand{\thesection}{\Alph{section}}

\begin{appendices}

\section{Homotopy Pullbacks} \label{sec:hopullbacks}

In this appendix, we gather a few facts about homotopy pullbacks in a fibration category.
These are well-known in folklore and we claim no originality in this work.

\subsection{Commutative squares in a fibration category}
   \begin{definition}\label{def:hopullback-1}
    A square
    \[ \begin{tikzcd}
        U \ar[r] \ar[d] & X \ar[d] \\
        V \ar[r] & Y
    \end{tikzcd} \]
    in a fibration category $\cat{C}$ is a \emph{homotopy pullback} if given any factorization $X \to \widetilde{X} \to Y$, the induced map $U \to \widetilde{X} \times_Y V$ is a weak equivalence.
   \end{definition}
   
   \begin{lemma} \label{lem:hopb-independent-factorization}
    \cref{def:hopullback-1} does not depend on the choice of factorization.
   \end{lemma}
   \begin{proof}
    Any two factorizations of $X \to Y$ can be connected by a zigzag of weak equivalences since, given a diagram,
    \[ \begin{tikzcd}
        & X \ar[ld, "\sim"'] \ar[rd, "\sim"] & \\
        \tilde{X} \ar[rd, two heads] && \tilde{X}' \ar[ld, two heads] \\
        & Y &
    \end{tikzcd} \]
    the induced map $X \to \widetilde{X} \times_Y \widetilde{X}'$ to the pullback may be factored into a weak equivalence followed by a fibration. Thus, it suffices to show that if two factorizations are connected by a weak equivalence
    \[ \begin{tikzcd}
        & X \ar[ld, "\sim"'] \ar[rd, "\sim"] & \\
        \tilde{X} \ar[rr, "\sim"] \ar[rd, two heads] && \tilde{X}' \ar[ld, two heads] \\
        & Y &
    \end{tikzcd} \]
    then the map $U \to \widetilde{X} \times_{Y} V$ is a weak equivalence if and only if the map $U \to \widetilde{X}' \times_{Y} V$ is. This follows by exactness of the pullback functor (\cref{lem:pb-is-exact}), Ken Brown's lemma, and two-out-of-three.
   \end{proof}
   
   \begin{lemma} \label{lem:hopb-independent-map}
    \cref{def:hopullback-1} does not depend on the choice of morphism for factorization. That is, a square
    \[ \begin{tikzcd}
        U \ar[r] \ar[d] & X \ar[d] \\
        V \ar[r] & Y
    \end{tikzcd} \] 
    is a homotopy pullback if and only if for any factorization $V \to \widetilde{V} \to Y$ of $V \to Y$ into a weak equivalence followed by a fibration, the induced map $U \to \widetilde{V} \times_Y X$ is a weak equivalence.
   \end{lemma}
   \begin{proof}
    It suffices to show that one of the dotted arrows in the diagram
    \[ \begin{tikzcd}
        U \ar[r, dotted] \ar[d, dotted] & \bullet \ar[r, two heads] \ar[d, "\sim"] \ar[rd, phantom, "\ulcorner" very near start] & X \ar[d, "\sim"] \\
        \bullet \ar[r, "\sim"] \ar[d, two heads] \ar[rd, phantom, "\ulcorner" very near start] & \bullet \ar[r, two heads] \ar[d, two heads] \ar[rd, phantom, "\ulcorner" very near start] & \bullet \ar[d, two heads] \\
        V \ar[r, "\sim"] & \bullet \ar[r, two heads] & Y
    \end{tikzcd} \]
    is a weak equivalence if and only if the other one is. This is an immediate consequence of two-out-of-three.
   \end{proof}
   
   \begin{lemma}\label{lem:invariance-under-weq}
    Given a commutative cube,
    \[ \begin{tikzcd}
        U_0 \ar[rr] \ar[rd, "\sim"] \ar[dd] && X_0 \ar[rd, "\sim"] \ar[dd] & {} \\
        {} & U_1 \ar[rr, crossing over] && X_1 \ar[dd] \\
        V_0 \ar[rd, "\sim"] \ar[rr] && Y_0 \ar[rd, "\sim"] & {} \\
        {} & V_1 \ar[rr] \ar[from=uu, crossing over] && Y_1
    \end{tikzcd} \]
    the front square is a homotopy pullback if and only if the back square is.
   \end{lemma}
   \begin{proof}
    Let $X_1 \xrightarrow{\sim} \tilde{X_1} \twoheadrightarrow Y_1$ be any factorization of $X_1 \to Y_1$ and let $X_0 \xrightarrow{\sim} \tilde{X_0} \twoheadrightarrow \tilde{X_1} \times_{Y_1} Y_0$ be a factorization of the induced map $X_0 \to \tilde{X_1} \times_{Y_1} Y_0$. By \cref{lem:hopb-independent-factorization}, it suffices to show the map $U_0 \to V_0 \times_{Y_0} \tilde{X_0}$ in
    \[ \begin{tikzcd}
        U_0 \ar[rr] \ar[rd, "\sim"] \ar[dd] && X_0 \ar[rd, "\sim"] \ar[dd] & {} \\
        & U_1 \ar[rr, crossing over] && X_1 \ar[dd] \\
        V_0 \times_{Y_0} \tilde{X}_0 \ar[rd, "\sim"] \ar[rr] && \tilde{X}_0 \ar[rd, "\sim"] & {} \\
        & V_1 \times_{Y_1} \tilde{X}_1 \ar[rr] \ar[from=uu, crossing over] && \tilde{X}_1
    \end{tikzcd} \]
   is a weak equivalence if and only if the map $U_1 \to V_1 \times_{Y_1} \tilde{X_1}$ is. By the gluing lemma, the map $V_0 \times_{Y_0} \tilde{X_0} \to V_1 \times_{Y_1} \tilde{X_1}$ is a weak equivalence, hence the result follows by two-out-of-three.
   \end{proof}
   
   \begin{corollary}
    A square in a fibration category $\cat{C}$ is a homotopy pullback if and only if it can be connected by a zigzag of natural weak equivalences (in the category of commutative squares in $\cat{C}$) with a pullback along a fibration.
   \end{corollary}
   \begin{proof}
    Follows from \cref{lem:invariance-under-weq}.
   \end{proof}
   \begin{theorem} \label{thm:htpy_pb_lemma}
    Consider the following diagram of maps between fibrant objects
    \[ \begin{tikzcd}
        A \ar[r] \ar[d] & B \ar[r] \ar[d] & C \ar[d] \\
        D \ar[r] & E \ar[r] & F
    \end{tikzcd} \]
    where the right square is a homotopy pullback.
    We have that the left square is a homotopy pullback if and only if the outer rectangle is.
\end{theorem}
\begin{proof}
    Factor $C \to F$ as a weak equivalence followed by a fibration.
    Successive pullbacks give the following diagram of maps.
    \[ \begin{tikzcd}
        A \ar[r] \ar[d] & B \ar[r] \ar[d] & C \ar[d, "\sim"] \\
        P' \ar[r] \ar[d] \ar[rd, phantom, "\ulcorner" very near start] & P \ar[r] \ar[d] \ar[rd, phantom, "\ulcorner" very near start] & \tilde{C} \ar[d, two heads] \\
        D \ar[r] & E \ar[r] & F
    \end{tikzcd} \]
    Observe the map $B \to P$ is a weak equivalence by \cref{lem:hopb-independent-factorization}.
    The map $P \to E$ is a fibration as it is a pullback of a fibration.
    Thus $B \to P \to E$ is a factorization of the map $B \to E$ as a weak equivalence followed by a fibration.
    Thus the left square and outer rectangles are each homotopy pullbacks if and only if the map $A \to P'$ is a weak equivalence.
\end{proof}

\subsection{Homotopy-commutative squares in a fibration category}
We now extend the notion of homotopy pullbacks to homotopy-commutative squares, under two additional assumptions on our ambient fibration category $\cat{C}$:
\begin{enumerate}
  \item it has a functorial choice of path object factorizations;
  \item the notions of homotopy and right homotopy coincide.
\end{enumerate}
These assumptions are not strictly required, however they are satisfied in many cases of interests, including the fibration categories of cubical and simplicial Kan complexes.
For instance, the path objects in cubical Kan complexes are given either by the functor $\lpath{\uvar}$ or equivalently, in light of \cref{rem:htpy_def}, by the functor $\rpath{\uvar}$.
The second assumption is satisfied for instance when all objects of $\cat{C}$ are cofibrant, i.e.~admit the left lifting property against maps that are both a fibration and a weak equivalence.

These assumptions make our arguments simpler.
This is an important gain, as the details of the general case, where path objects need to be replaced and homotopies require taking iterated cofibrant replacements, become unwieldy.
In the end, as evidenced by the preceding section, the language of fibration categories is well-equipped to handle strictly commuting squares, but it is limited when it comes to homotopy-commutative squares.
The latter are best handled in the language of $(\infty,1)$-categories, cf.~\cite{joyal:theory,lurie:htt}, which we choose not to invoke here.

\begin{definition}
    A \emph{homotopy-commutative square} is a (non-commutative) square of maps
    \[ \begin{tikzcd}
        A \ar[r, "f"] \ar[d, "h"'] & B \ar[d, "g"] \\
        C \ar[r, "k"'] & D
    \end{tikzcd} \]
    with a homotopy $H \from A \to \pob{D}$ from $gf$ to $kh$, denoted
    \[ \begin{tikzcd}
        A \ar[r, "f"] \ar[d, "h"'] & B \ar[d, ""{name=s, left}, "g"] \\
        C \ar[r, ""{name=t, above}, "k"'] & D \ar[from=s, to=t, bend right, "H"']
    \end{tikzcd} \]
\end{definition}
Recall that, for a co-span $(f \from A \to C, g \from B \to C)$, the \emph{double mapping path space} $P(f, g)$ is the pullback
\[ \begin{tikzcd}
    P(f, g) \ar[r] \ar[d] \ar[rd, phantom, "\ulcorner" very near start] & \pob{C} \ar[d, two heads] \\
    A \times B \ar[r, "f \times g"] & C \times C
\end{tikzcd} \]
of the path space fibration $\pob{C} \twoheadrightarrow C \times C$ along the map $f \times g \from A \times B \to C \times C$.
A homotopy-commutative square
\[ \begin{tikzcd}
    X \ar[r, "p_1"] \ar[d, "p_2"'] & A \ar[d, "f" ,""{name=s, left}] \\
    B \ar[r, "g"', ""{name=t, above}] & C \ar[from=s, to=t, bend right, "H"']
\end{tikzcd} \]
induces a commutative square,
\[ \begin{tikzcd}
    X \ar[r, "H"] \ar[d, "{(p_1, p_2)}"'] & \pob{C} \ar[d, two heads] \\
    A \times B \ar[r, "f \times g"] & C \times C
\end{tikzcd} \]
hence a canonical map $X \to P(f, g)$ into the double mapping path space.

\begin{definition} \label{def:weak-hopb}
    A homotopy-commutative square
    \[ \begin{tikzcd}
        X \ar[r] \ar[d] & A \ar[d, ""{name=s, left}, "f"] \\
        B \ar[r, ""{name=t, above}, "g"'] & C \ar[from=s, to=t, bend right, "H"']
    \end{tikzcd} \]
    is a \emph{homotopy pullback} if the canonical map $X \to P(f, g)$ is a weak equivalence.
\end{definition}

Note that at this point, we have introduced two a priori different definitions of homotopy pullbacks for commutative squares: \cref{def:hopullback-1} and \cref{def:weak-hopb} with the constant homotopy.
We now show that the two definitions agree.

\begin{theorem} \label{thm:hopb-def-equiv}
    A commutative square is a homotopy pullback as in \cref{def:hopullback-1} if and only if it is a homotopy pullback as in \cref{def:weak-hopb} (with the constant homotopy).
\end{theorem}

The proof will be preceded by a technical lemma.

\begin{lemma} \label{thm:weak-hopb-fib}
    Any pullback along two fibrations is a homotopy pullback (with the constant homotopy).
\end{lemma}
\begin{proof}
    Fix a pullback
    \[ \begin{tikzcd}
        P \ar[r, "p_1"] \ar[d, "p_2"'] \ar[rd, phantom, "\ulcorner" very near start] & A \ar[d, "f", two heads] \\
        B \ar[r, "g", two heads] & C
    \end{tikzcd} \]
    along fibrations $f \from A \twoheadrightarrow C$ and $g \from B \twoheadrightarrow C$.
    Post-composing the map $P \to C$ with the section of the path object $C \to \pob{C}$ gives the constant homotopy $P \to \pob{C}$ from $fp_1$ to $gp_2$.
    The outer square in
    \[ \begin{tikzcd}
        P \ar[r] \ar[d] & C \ar[d, "\sim"] \\
        P(f, g) \ar[r] \ar[d] \ar[rd, phantom, "\ulcorner" very near start] & \pob{C} \ar[d, two heads] \\
        A \times B \ar[r, "f \times g", two heads] & C \times C
    \end{tikzcd} \]
    is a pullback since $P$ is a pullback of $f$ and $g$.
    The bottom square is a pullback by definition, hence the top square is a pullback.
    As $f \times g$ is a fibration, the map $P(f, g) \to \pob{C}$ is a fibration.
    The map $C \to \pob{C}$ is a weak equivalence, hence the map $P \to P(f, g)$ is by right properness.
\end{proof}

\begin{proof}[Proof of \cref{thm:hopb-def-equiv}]
    We fix a commutative square and factor the right and bottom maps as weak equivalences followed by fibrations.
    \[ \begin{tikzcd}
        P \ar[rr, "p_1"] \ar[dd, "p_2"'] & {} & A \ar[d, "w", "\sim"'] \\
        {} & {} & A' \ar[d, two heads, "f'"] \\
        B \ar[r, "v", "\sim"'] & B' \ar[r, two heads, "g'"] & C
    \end{tikzcd} \]
    Taking the pullback of $f'$ along $g'$ gives the following diagram of maps.
    \[ \begin{tikzcd}
        P \ar[rr, "p_1"] \ar[dd, "p_2"'] \ar[rd, dotted] & {} & A \ar[d, "w", "\sim"'] \\
        {} & P' \ar[r, "p'_1"] \ar[d, "p'_2"'] \ar[rd, phantom, "\ulcorner" very near start] & A' \ar[d, two heads, "f'"] \\
        B \ar[r, "v", "\sim"'] & B' \ar[r, two heads, "g'"] & C
    \end{tikzcd} \]
    By right properness and two-out-of-three, $P$ is a homotopy pullback as in \cref{def:hopullback-1} if and only if the map $P \to P'$ is a weak equivalence.
    Commutativity of this diagram gives that the diagram
    \[ \begin{tikzcd}[column sep = small]
        P \ar[rrd, bend left] \ar[rdd, "{(wp_1, vp_2)}"', bend right] \ar[rd] & {} & {} & {} \\
        {} & P' \ar[r] \ar[d, "{(p'_1, p'_2)}"' near start] & C \ar[r] & \pob{C} \ar[d, two heads] \\
        {} & A' \times B' \ar[rr, "f' \times g'"] & {} & C \times C
    \end{tikzcd} \]
    commutes.
    This induces a commutative triangle
    \[ \begin{tikzcd}
        P \ar[rr] \ar[rd] & {} & P(f, g) \\
        {} & P' \ar[ur] & {}
    \end{tikzcd} \]
    where the right map is a weak equivalence by \cref{thm:weak-hopb-fib}.
    By two-out-of-three, the map $P \to P(f, g)$ is a weak equivalence if and only if $P \to P'$ is.
\end{proof}

It is often simpler to work with commutative squares than homotopy-commutative squares.
We establish some definitions and results which characterize when a homotopy-commutative square is a homotopy pullback using commutative squares.

Given a morphism $f \from A \to B$, recall the Factorization Lemma of \cite{brown} uses the double mapping path space $P(f, \id[B])$ to contruct a factorization of $f$ as a section $\FLsect{f} \from A \xrightarrow{\sim} P(f, \id[B])$ of an acyclic fibration followed by a fibration $\FLfib{f} \from P(f, \id[B]) \to B$.
One can equivalently use the double mapping path space $P(\id[B], f)$, yielding a weak equivalence $\FLsect{f} \from A \to P(\id[B], f)$ and a fibration $\FLfibr{f} \from P(\id[B], f) \to B$.

A co-span $(f \from A \to C, g \from B \to C)$ induces maps $g^* \from P(f, g) \to P(f, \id[C])$ and $f^* \from P(f, g) \to P(\id[C], g)$ via the squares,
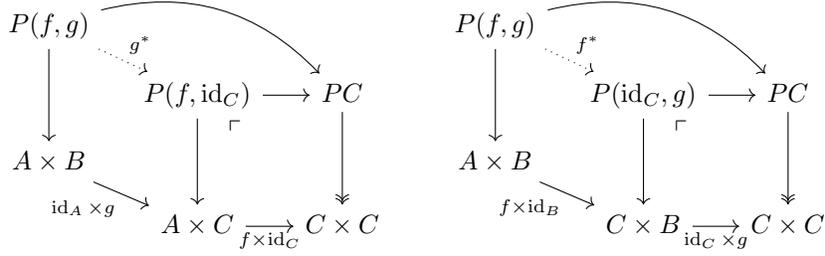
\[ \begin{tikzcd}[sep = small]
    P(f, g) \ar[rrd, bend left] \ar[dd] \ar[rd, dotted, "{g^*}"] & {} & {} \\
    {} & P(f, \id[C]) \ar[r] \ar[dd] \ar[rdd, phantom, "\ulcorner" very near start] & \pob{C} \ar[dd, two heads] \\
    A \times B \ar[rd, "{\id[A] \times g}"'] & {} & {} \\
    {} & A \times C \ar[r, "{f \times \id[C]}"'] & C \times C
\end{tikzcd} \qquad \begin{tikzcd}[sep = small]
    P(f, g) \ar[rrd, bend left] \ar[dd] \ar[rd, dotted, "{f^*}"] & {} & {} \\
    {} & P(\id[C], g) \ar[r] \ar[dd] \ar[rdd, phantom, "\ulcorner" very near start] & \pob{C} \ar[dd, two heads] \\
    A \times B \ar[rd, "{f \times \id[B]}"'] & {} & {} \\
    {} & C \times B \ar[r, "{\id[C] \times g}"'] & C \times C
\end{tikzcd} \]
respectively.
\begin{proposition}
    Let $(f \from A \to C, g \from B \to C)$ be a co-span.
    \begin{enumerate}
        \item If $g$ is a weak equivalence then $g^* \from P(f, g) \to P(f, \id[C])$ is.
        \item If $f$ is a weak equivlanece then $f^* \from P(f, g) \to P(\id[C], g)$ is.
    \end{enumerate}
\end{proposition}
\begin{proof}
    By the two-pullback lemma, the squares
    
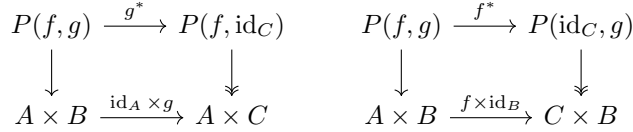
\[ \begin{tikzcd}
        P(f, g) \ar[r, "{g^*}"] \ar[d] & P(f, \id[C]) \ar[d, two heads] \\
        A \times B \ar[r, "{\id[A] \times g}"] & A \times C
    \end{tikzcd} \qquad \begin{tikzcd}
        P(f, g) \ar[r, "{f^*}"] \ar[d] & P(\id[C], g) \ar[d, two heads] \\
        A \times B \ar[r, "{f \times \id[B]}"] & C \times B
    \end{tikzcd} \]
    are pullbacks.
    The result follows from right properness.
\end{proof}
For a co-span $(f \from A \to C, g \from B \to C)$ and a map $h \from C \to D$, define the map $h_* \from P(f, g) \to P(hf, hg)$ to be the map induced by the following diagram.
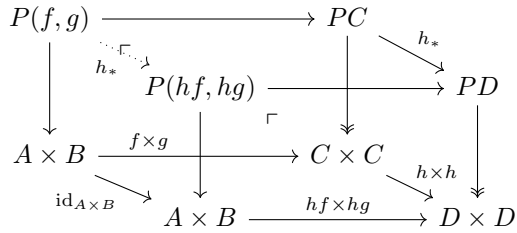
\[ \begin{tikzcd}[sep = small]
    P(f, g) \ar[rr] \ar[dd] \ar[rrdd, phantom, "\ulcorner" very near start] \ar[rd, dotted, "{h_*}"'] & {} & \pob{C} \ar[dd, two heads] \ar[rd, "h_*"] & {} \\
    {} & P(hf, hg) \ar[rr] \ar[dd] \ar[rrdd, phantom, "\ulcorner" very near start] & {} & \pob{D} \ar[dd, two heads] \\
    A \times B \ar[rr, "f \times g" near start] \ar[rd, "{\id[A \times B]}"'] & {} & C \times C \ar[rd, "h \times h"] & {} \\
    {} & A \times B \ar[rr, "hf \times hg"] & {} & D \times D
\end{tikzcd} \]
This construction gives a compatibility between factorizations in a fibration category and composition, which we express formally in the following proposition.
\begin{proposition} \label{thm:FL_comp}
    For $f \from X \to Y$ and $g \from Y \to Z$,
    \begin{enumerate}
        \item the diagrams
        \[ \begin{tikzcd}[sep = small]
            {} & {} & P(gf, \id[Z]) \ar[ddrr, two heads, "\FLfib{(gf)}"] & {} & {} \\
            {} & P(f, \id[Y]) \ar[dr, two heads, "\FLfib{f}"] \ar[ur, "g^*g_*"'] & {} & {} & {} \\
            X \ar[rr, "f"] \ar[ur, "\sim"'] \ar[uurr, "\sim", bend left] & {} & Y \ar[rr, "g"] & {} & Z
        \end{tikzcd} \qquad \begin{tikzcd}[sep = small]
            {} & {} & P(\id[Z], gf) \ar[ddrr, two heads, "\FLfibr{(gf)}"] & {} & {} \\
            {} & P(\id[Y], f) \ar[dr, two heads, "\FLfibr{f}"] \ar[ur, "g^*g_*"'] & {} & {} & {} \\
            X \ar[rr, "f"] \ar[ur, "\sim"'] \ar[uurr, "\sim", bend left] & {} & Y \ar[rr, "g"] & {} & Z
        \end{tikzcd} \]
        commute;
        \item the maps $g^*g_* \from P(f, \id[Y]) \to P(gf, \id[Z])$ and $g^*g_* \from P(\id[Y], f) \to P(\id[Y], gf)$ are weak equivalences.
    \end{enumerate}
\end{proposition}
\begin{proof}
    Item (1) follows from the construction of the maps $g^*$ and $g_*$. Item (2) follows by two-out-of-three.
\end{proof}
\begin{definition}
    For a homotopy-commutative square,
    \[ \begin{tikzcd}
        X \ar[r, "p_1"] \ar[d, "p_2"'] & A \ar[d, "f" ,""{name=s, left}] \\
        B \ar[r, "g"', ""{name=t, above}] & C \ar[from=s, to=t, bend right, "H"']
    \end{tikzcd} \]
    we define the maps $\rdrop{H} \from X \to P(f, \id[C])$ and $\ldrop{H} \from X \to P(\id[C], g)$ by post-composing the canonical map $X \to P(f, g)$ with the maps $g^* \from P(f, g) \to P(f, \id[C])$ and $f^* \from P(f, g) \to P(\id[C], g)$, respectively.
\end{definition}
\begin{theorem} \label{thm:hopb-equiv}
    A homotopy-commutative square
    \[ \begin{tikzcd}
        X \ar[r, "p_1"] \ar[d, "p_2"'] & A \ar[d, "f" ,""{name=s, left}] \\
        B \ar[r, "g"', ""{name=t, above}] & C \ar[from=s, to=t, bend right, "H"']
    \end{tikzcd} \]
    is a homotopy pullback if and only if any of the commutative squares
    \[ \begin{tikzcd}
        X \ar[r, "H"] \ar[d, "{(p_1, p_2)}"'] \ar[rd, phantom, "{(1)}" description] & \pob{C} \ar[d, two heads] \\
        A \times B \ar[r, "f \times g"'] & C \times C
    \end{tikzcd} \qquad \begin{tikzcd}
        X \ar[r, "\rdrop{H}"] \ar[d, "p_2"'] \ar[rd, phantom, "{(2)}" description] & P(f, \id[C]) \ar[d, two heads, "\FLfib{f}"] \\
        B \ar[r, "g"'] & C
    \end{tikzcd} \qquad \begin{tikzcd}
        X \ar[r, "p_1"] \ar[d, "\ldrop{H}"'] \ar[rd, phantom, "{(3)}" description] & A \ar[d, "f"] \\
        P(\id[C], g) \ar[r, two heads, "{\FLfibr{g}}"'] & C
    \end{tikzcd} \]
    is a homotopy pullback.
\end{theorem}
\begin{proof}
    Square (1) follows by \cref{thm:hopb-def-equiv,lem:hopb-independent-factorization}.
    We show square (2) is equivalent, as square (3) follows analogously.

    As $f = \FLfib{f}\FLsect{f}$, there is a canonical map $P(f, g) \to P(\FLfib{f}, g)$ which is a weak equivalence by the gluing lemma.
    \[ \begin{tikzcd}[sep = small]
        P(f, g) \ar[rr] \ar[dd] \ar[rd, dotted] \ar[rrdd, phantom, "\ulcorner" very near start] & {} & \pob{C} \ar[dd, two heads] \ar[rd, "{\id[\pob{C}]}", "\sim"'] & {} \\
        {} & P(\FLfib{f}, g) \ar[rr] \ar[dd] \ar[rrdd, phantom, "\ulcorner" very near start] & {} & \pob{C} \ar[dd, two heads] \\
        A \times B \ar[rr, "f \times g" near start] \ar[rd, "{\FLsect{f} \times \id[B]}"', "\sim"] & {} & C \times C \ar[rd, "{\id[C \times C]}", "\sim"'] & {} \\
        {} & P(f, \id[C]) \times B \ar[rr, "{\FLfib{f} \times g}"] & {} & C \times C
    \end{tikzcd} \]
    The result then follows by two-out-of-three.
\end{proof}

\subsection{Homotopy-commutative squares of cubical Kan complexes}
We consider the fibration category $\Kan$ of cubical Kan complexes, which in particular satisfies the assumptions of the previous subsection with the path object functor given by $\rpath{\uvar} \from \Kan \to \Kan$ and the double mapping path space defined analogously.
We now prove additional results about homotopy pullbacks, primarily making use of the Kan operation to describe composition of homotopies.

These results are true in greater generality --- in particular, it is possible to work directly with path objects in fibration categories to prove statements analogous to, e.g.~\cref{thm:htpy_two_pb}.
We choose not to do this, since, as indicated before, the language of fibration categories is somewhat ill-suited to work with homotopy-commutative squares.

\begin{theorem} \label{thm:hopb_htpy}
    Let
    \[ \begin{tikzcd}
        P \ar[r, "p_1"] \ar[d, "p_2"'] & A \ar[d, "f" ,""{name=s, left}] \\
        B \ar[r, "g"', ""{name=t, above}] & C \ar[from=s, to=t, bend right, "H"']
    \end{tikzcd} \qquad \begin{tikzcd}
        P \ar[r, "q_1"] \ar[d, "q_2"'] & A \ar[d, "f" ,""{name=s, left}] \\
        B \ar[r, "g"', ""{name=t, above}] & C \ar[from=s, to=t, bend right, "K"']
    \end{tikzcd} \]
    be homotopy-commutative squares.
    Given homotopies $\alpha \from p_1 \sim q_1$, $\beta \from p_2 \sim q_2$, and a map $\gamma \from P \to \rhom(\cube{2}, C)$ such that
    \[ \begin{array}{l l}
        \face{*}{1,0}\gamma  = H & \face{*}{1,1} \gamma = K, \\
        \face{*}{2,0}\gamma = f_* \alpha & \face{*}{2,1}\gamma = g_* \beta
    \end{array} \]
    the homotopy-commutative square
        \[ \begin{tikzcd}
            P \ar[r, "\alpha"] \ar[d, "\beta"'] & \rpath{A} \ar[d, "f_*" ,""{name=s, left}] \\
            \rpath{B} \ar[r, "g_*"', ""{name=t, above}] & \rpath{C} \ar[from=s, to=t, bend right, "\gamma"']
        \end{tikzcd} \]
    is a homotopy pullback if and only if the left or right square is.
    In particular, the left square is a homotopy pullback if and only if the right square is.
\end{theorem}
\begin{proof}
    As $\rpath{\uvar} \from \cSet \to \cSet$ is a right adjoint, the square
    \[ \begin{tikzcd}
        \rpath{P(f, g)} \ar[r] \ar[d] & \rhom(\cube{2}, C) \ar[d, "{(\face{*}{2,0}, \face{*}{2,1})}"] \\
        \rpath{A} \times \rpath{B} \ar[r, "{f_* \times g_*}"] & \rpath{C} \times \rpath{C}
    \end{tikzcd} \]
    is a pullback.
    By assumption, $\gamma$ and $(\alpha, \beta)$ induce a map into the pullback $P \to \rpath{P(f, g)}$.
    This then follows by two-out-of-three.
\end{proof}

\begin{definition}
    Let $f, g, h \from X \to Y$ be cubical maps and $H, K \from X \to \rpath{Y}$ be homotopies from $f$ to $g$ and from $g$ to $k$, respectively.
    \begin{enumerate}
        \item A \emph{concatenation square} for $H$ and $K$ is a map $\alpha \from X \to \rhom(\cube{2}, X)$ such that
        \[ \begin{array}{l l}
            \face{*}{1,0}\alpha = H & \face{*}{1,1}\alpha = \degen{*}{1}h \\
            {} & \face{*}{2,1}\alpha = K.
        \end{array} \qquad \begin{tikzcd}
            f \ar[r] \ar[d, "H"'] & h \ar[d, equal] \\
            g \ar[r, "K"] & h
        \end{tikzcd} \]
        \item A \emph{concatenation} of $H$ and $K$ is a map $L \from X \to \rpath{X}$ such that $L = \face{*}{2,0}\alpha$ for some concatenation square $\alpha$ for $H$ and $K$.
    \end{enumerate}
\end{definition}
Note that a concatenation always exists by transitivity of homotopy.

\begin{proposition} \label{thm:htpy_conc_unique}
    For maps $f, g, h \from X \to Y$ and homotopies $H \from f \sim g$ and $K \from g \sim k$, if $L, L' \from X \to \rpath{Y}$ are concatenations for $H$ and $K$ then there is a homotopy $\eta \from X \to \rhom(\cube{2}, Y)$ from $L$ to $L'$ such that
    \[ \face{*}{2,0}\eta = \degen{*}{1}h, \quad \face{*}{2,1}\eta = \degen{*}{1}k. \]
\end{proposition}
\begin{proof}
    Identifying homotopies of the form $X \to \rpath{Y}$ with homotopies of the form $X \gprod \cube{1} \to Y$, this follows from \cref{thm:htpy_obox_general}.
\end{proof}
If $H \from X \to \rpath{Y}$ is a homotopy from a map $f \from X \to Y$ to $g \from X \to Y$ then the maps $\conn{*}{1,0}H, \degen{*}{2}H \from X \to \rhom(\cube{2}, Y)$ exhibit $H$ as both the concatenation of $H$ with $\degen{*}{1}f$ and of $\degen{*}{1}g$ with $H$, respectively.

Recall the cubical set $\Q$ is given by the following pushout.
\[ \begin{tikzcd}
    \cube{1} \ar[r, "{\face{}{1,1}}"] \ar[d] \ar[rd, phantom, "\ulcorner" very near end] & \cube{2} \ar[d] \\
    \cube{0} \ar[r] & \Q
\end{tikzcd} \]
We write $\twosp$ for the 2-spine, i.e.~the cubical set obtained by the following pushout.
\[ \begin{tikzcd}[column sep = large]
    \cube{0} \sqcup \cube{0} \ar[r, "{\face{}{1,1} \sqcup \face{}{1,0}}"] \ar[d] \ar[rd, phantom, "\ulcorner" very near end] & \cube{1} \sqcup \cube{1} \ar[d] \\
    \cube{0} \ar[r] & \twosp
\end{tikzcd} \]
By left properness, both $\Q$ and $\twosp$ are contractible, thus the maps 
\[ [\face{}{1,0}, \face{}{2,1}] \from \twosp \to \Q, \quad \face{}{2,0} \from \cube{1} \to \Q \] 
are acyclic cofibrations.
For any Kan complex $X$, the pre-composition maps 
\[ [\face{}{1,0}, \face{}{2,1}]^* \from \rhom(\Q, X) \to \rhom(\twosp, X), \quad \face{*}{2,0} \from \rhom(\Q, X) \to \rhom(\cube{1}, X) \]
are acyclic fibrations by \cref{thm:pp_anodyne}.
In particular, the map $[\face{}{1,0}, \face{}{2,1}]^* \from \rhom(\Q, X) \to \rhom(\twosp, X)$ has a section $\rhom(\twosp, X) \to \rhom(\Q, X)$.
A choice of section $\rhom(\twosp, X) \to \rhom(\Q, X)$ corresponds to a choice of concatenation for every (compatible) pair of homotopies in $X$.
For a chosen section, the composite
\[ \begin{tikzcd}[column sep = large]
    \rhom(\twosp, X) \ar[r, "{[\face{}{1,0}, \face{}{2,1}]^*}", "\sim"', hook] & \rhom(\Q, X) \ar[r, two heads, "\sim"', "{\face{*}{2,0}}"] & \rpath{X}
\end{tikzcd} \]
returns the chosen concatenation for every (compatible) pair of homotopies in $X$.
This map is always a homotopy equivalence.

\begin{lemma} \label{thm:conc_pb}
    Let
    \[ \begin{tikzcd}
        X \ar[r, "p_1"] \ar[d, "p_2"'] & A \ar[d, "f" ,""{name=s, left}] \\
        B \ar[r, "g"', ""{name=t, above}] & C
    \end{tikzcd} \]
    be a non-commutative square.
    For a map $h \from X \to C$ with homotopies $H \from fp_1 \sim h$ and $K \from h \sim gp_2$, the following are equivalent:
    \begin{enumerate}
        \item for any concatenation $L$ of $H$ and $K$, the square
        \[ \begin{tikzcd}
            X \ar[r, "p_1"] \ar[d, "p_2"'] & A \ar[d, "f" ,""{name=s, left}] \\
            B \ar[r, "g"', ""{name=t, above}] & C \ar[from=s, to=t, bend right, "L"']
        \end{tikzcd} \]
        is a homotopy pullback 
        \item there exists a concatenation $L$ of $H$ and $K$ such that the square
        \[ \begin{tikzcd}
            X \ar[r, "p_1"] \ar[d, "p_2"'] & A \ar[d, "f" ,""{name=s, left}] \\
            B \ar[r, "g"', ""{name=t, above}] & C \ar[from=s, to=t, bend right, "L"']
        \end{tikzcd} \]
        is a homotopy pullback
        \item the map $X \to P(f, \id[C]) \times_{C} P(\id[C], g)$ induced by $\rdrop{H} \from X \to P(f, \id[C])$ and $\ldrop{K} \from X \to P(\id[C], g)$ via the square
        \[ \begin{tikzcd}
            X \ar[rrd, bend left, "{\rdrop{H}}"] \ar[rdd, bend right, "\ldrop{K}"'] \ar[rd, dotted] & {} & {} \\
            {} & {P(f, \id[C]) \times_C P(\id[C], g)} \ar[r] \ar[d] \ar[rd, phantom, "\ulcorner" very near start] & {P(f, \id[C])} \ar[d, two heads, "\FLfib{f}"] \\
            {} & {P(\id[C], g)} \ar[r, two heads, "\FLfibr{g}"] & C
        \end{tikzcd} \]
        is a weak equivalence.
    \end{enumerate}
\end{lemma}
\begin{proof}
    Equivalence of (1) and (2) follows from \cref{thm:hopb_htpy,thm:htpy_conc_unique}.
    
    We have that the square
    \[ \begin{tikzcd}
        P(f, \id[C]) \times_C P(\id[C], g) \ar[d, two heads] \ar[r] \ar[rd, phantom, "\ulcorner" very near start] & \rhom(\twosp, C) \ar[d, two heads] \\
        A \times B \ar[r] & C \times C
    \end{tikzcd} \]
    is a pullback.
    A concatenation map $\rhom(\twosp, C) \to \rpath{C}$ induces a weak equivalence 
    \[ P(f, \id[C]) \times_C P(\id[C], g) \xrightarrow{\sim} P(f, g)\] 
    by the gluing lemma.
    \[ \begin{tikzcd}
        P(f, \id[C]) \times_C P(\id[C], g) \ar[rr] \ar[dd] \ar[rd, dotted, "\sim"] & {} & \rhom(\twosp, C) \ar[dd, two heads] \ar[rd, "\sim"] & {} \\
        {} & P(f, g) \ar[rr] \ar[dd] & {} & \rpath{C} \ar[dd, two heads] \\
        A \times B \ar[rr] \ar[rd, "{\id[A \times B]}"'] & {} & C \times C \ar[rd, "{\id[C \times C]}"] & {} \\
        {} & A \times B \ar[rr] & {} & C \times C
    \end{tikzcd} \]
    The composite $X \to P(f, \id[C]) \times_C P(\id[C], g) \to P(f, g)$ is the canonical map for some concatenation of $H$ and $K$.
    If $X \to P(f, \id[C]) \times_C P(\id[C], g)$ is a weak equivalence then this composite is, thus (3) implies (2).
    By two-out-of-three, we have that (1) implies (3).
\end{proof}
In light of \cref{thm:conc_pb}, we say \emph{the} concatenation of two homotopies is a homotopy pullback if any of the equivalent conditions are satisfied.
\begin{theorem} \label{thm:htpy_two_pb}
    Suppose the right square in
    \[ \begin{tikzcd}
        A \ar[r, "f"] \ar[d, "\alpha"'] & B \ar[r, "g"] \ar[d, "\beta", ""{left, name=l1}] & C \ar[d, "\gamma", ""{left, name=r1}] \\
        A' \ar[r, "f'"', ""{left, name=l2}] & B' \ar[r, "g'"', ""{left, name=r2}] & C' \ar[from=l1, to=l2, bend right, "G"'] \ar[from=r1, to=r2, bend right, "H"']
    \end{tikzcd} \]
    is a homotopy pullback.
    Then, the left square is a homotopy pullback if and only if the composite of $g'_*G$ with $Hf$ is.
\end{theorem}
\begin{proof}
    By \cref{thm:hopb-equiv}, the left square is a homotopy pullback if and only if square (1) in
    \[ \begin{tikzcd}
        A \ar[r, "f"] \ar[d, "{\ldrop{G}}"'] \ar[rd, phantom, "(1)" description] & B \ar[r, "{\rdrop{H}}"] \ar[d, "\beta"] \ar[rd, "(2)" description, phantom] & P(\gamma, \id[C']) \ar[d, "\FLfib{\gamma}", two heads] \\
        P(\id[B'], f') \ar[r, "\FLfibr{f'}"] \ar[d, "g'^*g'_*"', "\sim"] & B' \ar[r, "g'"] \ar[d, phantom, "(3)" description] & C' \ar[d, "{\id[C']}", two heads] \\
        P(\id[C'], g'f') \ar[rr, "\FLfibr{(g'f')}"', two heads] & {} & C'
    \end{tikzcd} \]
    is.
    By \cref{thm:htpy_pb_lemma}, this is equivalent to the composite square (1) and (2) being a homotopy pullback (since square (2) is a homotopy pullback by \cref{thm:hopb-equiv}).
    Square (3) is a homotopy pullback since $g'^*g'_*$ is a weak equivalence (by \cref{thm:FL_comp}). 
    Again, by \cref{thm:htpy_pb_lemma}, the composite square (1) and (2) is a homotopy pullback if and only if the outermost square is a homotopy pullback. 
    This is equivalent to the concatenation being a homotopy pullback by \cref{thm:FL_comp,thm:conc_pb}.
\end{proof}

\end{appendices}

\bibliographystyle{amsalphaurlmod}
\bibliography{all-refs.bib}

\end{document}